\def\draftdate{January 10, 2017}

\documentclass{amsart}
\usepackage{amssymb}
\usepackage{xy}
\usepackage{hyperref}

\newcommand{\supdot}{^{\bullet}}
\newcommand{\subdot}{_{\bullet}}

\let\iso\cong
\let\sma\wedge
\newcommand{\dR}{\mathbf{R}}

\newcommand{\putatop}[2]{\genfrac{}{}{0pt}{}{\scriptstyle #1}{\scriptstyle #2}}

\renewcommand{\to}{\mathchoice{\longrightarrow}{\rightarrow}{\rightarrow}{\rightarrow}}
\newcommand{\from}{\mathchoice{\longleftarrow}{\leftarrow}{\leftarrow}{\leftarrow}}

\newcommand{\overto}[1]{\xrightarrow{\,#1\,}}
\newcommand{\overfrom}[1]{\xleftarrow{\,#1\,}}

\DeclareMathAlphabet{\catsymbfont}{U}{rsfs}{m}{n}

\newcommand{\aA}{{\catsymbfont{A}}}
\newcommand{\aB}{{\catsymbfont{B}}}
\newcommand{\aC}{{\catsymbfont{C}}}
\newcommand{\Cat}{\mathop{\aC\!\mathrm{at}}\nolimits}
\newcommand{\aCat}{\Cat}
\newcommand{\aD}{{\catsymbfont{D}}}
\newcommand{\aE}{{\catsymbfont{E}}}
\newcommand{\aF}{{\catsymbfont{F}}}
\newcommand{\aG}{{\catsymbfont{G}}}
\newcommand{\aH}{{\catsymbfont{H}}}

\newcommand{\aM}{{\catsymbfont{M}}}
\newcommand{\aN}{{\catsymbfont{N}}}
\newcommand{\Mod}{\mathop {\aM\mathrm{od}}\nolimits}
\newcommand{\aO}{{\catsymbfont{O}}}
\newcommand{\LObj}{\aO\mathrm{bj}^{\ell}}

\newcommand{\aP}{{\catsymbfont{P}}}
\newcommand{\aQ}{{\catsymbfont{Q}}}
\newcommand{\aR}{{\catsymbfont{R}}}
\newcommand{\aS}{{\catsymbfont{S}}}

\newcommand{\BiSp}{\aB\aS}

\newcommand{\aW}{{\catsymbfont{W}}}
\newcommand{\aX}{{\catsymbfont{X}}}
\newcommand{\aY}{{\catsymbfont{Y}}}
\newcommand{\aZ}{{\catsymbfont{Z}}}

\newcommand{\bL}{{\mathbb{L}}}

\newcommand{\bR}{{\mathbb{R}}}
\newcommand{\bS}{{\mathbb{S}}}

\newcommand{\oA}{\mathcal{A}}
\newcommand{\Ass}{\mathop{\oA\mathrm{ss}}\nolimits}

\newcommand{\oD}{\mathcal{D}}
\newcommand{\oE}{\mathcal{E}}
\newcommand{\oEndb}{\mathop{\oE\mathrm{nd}^{b}}\nolimits}

\newcommand{\splus}{\Sigma^{\infty}_+}
\newcommand{\THC}{THC}
\newcommand{\CC}{CC}

\newcommand{\perf}{{\mathrm{Perf}}}

\def\quickop#1{\expandafter\DeclareMathOperator\csname
#1\endcsname{#1}}
\quickop{id}\quickop{Id}
\quickop{Tor}\quickop{Ext}
\quickop{Hom}\quickop{Tot}
\quickop{holim}\quickop{hocolim}\quickop{hofib}
\quickop{Ho}
\quickop{cy}
\quickop{colim}
\quickop{op}
\quickop{End}
\quickop{std}
\quickop{Cof}
\quickop{Cell}
\quickop{Zig}
\quickop{ex}

\numberwithin{equation}{section}

\newtheorem{thm}[equation]{Theorem}
\newtheorem{main}{Theorem}

\newtheorem{cor}[equation]{Corollary}

\newtheorem{lem}[equation]{Lemma}
\newtheorem{prop}[equation]{Proposition}
\theoremstyle{definition}
\newtheorem{defn}[equation]{Definition}
\newtheorem{defnsch}[equation]{Definition Schema}

\newtheorem{cons}[equation]{Construction}

\theoremstyle{remark}

\newtheorem{example}[equation]{Example}

\xyoption{arrow}
\xyoption{matrix}
\xyoption{cmtip}
\SelectTips{cm}{}

\newcommand{\term}[1]{\textit{#1}}

\newdir{ >}{{}*!/-5pt/\dir{>}}

\begin{document}

\title
{$E_{2}$ Structures and Derived Koszul Duality in String Topology}

\author{Andrew J. Blumberg}
\address{Department of Mathematics, The University of Texas,
Austin, TX \ 78712}
\email{blumberg@math.utexas.edu}
\thanks{The first author was supported in part by NSF grant
DMS-1151577}
\author{Michael A. Mandell}
\address{Department of Mathematics, Indiana University,
Bloomington, IN \ 47405}
\thanks{The second author was supported in part by NSF grant DMS-1505579}
\email{mmandell@indiana.edu}

\date{\draftdate} 

\begin{abstract}
We construct an equivalence of $E_{2}$ algebras between two models for
the Thom spectrum of the free loop space that are related by derived
Koszul duality.  To do this, we describe the functoriality and
invariance properties of topological Hochschild cohomology.
\end{abstract}

\maketitle


\section*{Introduction}

Chas and Sullivan started the subject of string topology with
their observation that the homology of the free loop space $LM$ of a
closed oriented manifold $M$ admits a Gerstenhaber structure that can
be defined geometrically in terms of natural operations on loops and
intersection of chains on the manifold.  Contemporaneously, the
solution to Deligne's Hochschild cohomology conjecture (by
Kontsevich-Soibelman~\cite{Kontsevich-Deligne},
McClure-Smith~\cite{McClureSmith-Deligne},
Tamarkin~\cite{Tamarkin-Deligne},
Voronov~\cite{Voronov-Deligne}, Berger-Fresse~\cite{BergerFresse-Deligne}, 
and perhaps others) established a Gerstenhaber
algebra structure on the Hochschild cohomology of a ring or
differential graded algebra or even $A_{\infty}$ ring spectrum.
Cohen-Jones~\cite{CohenJones-THC}, in the course of giving a homotopical interpretation of
the string topology product, connected these two ideas by relating a
certain Thom spectrum of $LM$ with the topological Hochschild
cohomology $\THC(DM)$ of the Spanier-Whitehead dual $DM$.  The homology
of $\THC(DM)$ is canonically isomorphic to the Hochschild cohomology of
the cochain algebra $C^{*}(M)$; Cohen-Jones~\cite{CohenJones-THC} in particular produces a
shifted isomorphism from the homology of the free loop space to the
homology of $\THC(DM)$ that takes the string topology product to the
cup product in Hochschild cohomology.  Later work of
Malm~\cite{Malm-Thesis} and
Felix-Menichi-Thomas~\cite{FelixMenichiThomas} give an isomorphism of
Gerstenhaber algebras.

The Felix-Menichi-Thomas work derives from work of
Keller~\cite{Keller-Picard} (see also~\cite{Keller-DIH}) building on
unpublished work of Buchweitz.  Keller~\cite{Keller-DIH} shows that
(under mild hypotheses) the Hochschild cochains of Koszul dual 
dg algebras are equivalent as
$E_{2}$ algebras (or, more specifically, $B_{\infty}$ algebras,
cf.~\cite{GerstenhaberVoronov,Young-Thesis}).
When $M$ is simply connected, the derived Koszul dual of $C^{*}M$ is
the cobar construction $\bar\Omega C_{*}M$, which is the Adams-Hilton
model for the chains on the based loop space $C_{*}(\Omega M)$;
Felix-Menichi-Thomas~\cite{FelixMenichiThomas} constructs an
isomorphism of Gerstenhaber algebras
\[
HH^{*}(C^{*}M)\iso HH^{*}(\bar\Omega C_{*}M).
\]
Since $HH^{*}(\bar\Omega C_{*}M)$ is isomorphic to the homology of
$\THC(\splus \Omega M)$,  
in spectral models, we should look for an equivalence of $E_{2}$ ring
spectra between $THC(DM)$ and $THC(\splus \Omega M)$.  Our main result
is the following theorem, proved in Section~\ref{sec:pfstring}.

\begin{main}\label{main:string}
Let $X$ be a simply connected finite cell complex; then $THC(DX)$ and $THC(\splus
\Omega X)$ are weakly equivalent as $E_{2}$ ring spectra.
\end{main}

Beyond the technical role of $HH^{*}(\bar\Omega C_{*}M)$ in the
comparison of Gerstenhaber algebra structures, the spectral
analogue $\THC(\splus \Omega M)$ in the previous theorem also provides a
connection between string topology and topological field theory, as 
explained in~\cite{BCT}.  Furthermore, \cite{BCT} sketches a
relationship between $\THC(\splus \Omega M)$ and the wrapped Fukaya
category of $T^{*}M$, motivated by the work of
Abbondandolo-Schwarz~\cite{AbbondandoloSchwarz} and
Abouzaid~\cite{Abouzaid}.  (See also~\cite{SeidelICM} for discussion of the
significance of Hochschild cohomology of Fukaya categories.)
Indeed, the previous theorem (and the machinery we develop to prove
it) fills in results stated in~\cite{BCT} but deferred to a future
paper.

In the discussion above and in the statement of
Theorem~\ref{main:string}, we are using $\THC$ to denote a derived
version of the topological Hochschild cohomology spectrum.  What this
means is slightly complicated by the fact the standard cosimplicial
construction is not functorial.  In the setting of differential graded
categories, Keller~\cite{Keller-Picard,Keller-DIH} made sense of this
for Hochschild cochains and proved limited functoriality and
invariance results for Hochschild cohomology.  Part of the purpose of
this paper is to provide a spectral version of this theory.

For the $E_{2}$ structure, we use the McClure-Smith theory
of~\cite{McClureSmith-Deligne,McClureSmith-CosimplicialCubes},
which establishes the action of a specific $E_{2}$ operad $\oD_{2}$ on
totalization ($\Tot$) of the topological Hochschild
cosimplicial construction of a strictly associative ring spectrum in
any modern category of spectra (such as symmetric spectra, orthogonal
spectra, or EKMM $S$-modules).  We denote this point-set topological
Hochschild cohomology construction as $\CC$ in the following theorem,
proved in Section~\ref{sec:pfwefunct}.

\begin{main}\label{main:wefunct}
Let $\aS$ denote either symmetric spectra, orthogonal spectra, or EKMM
$S$-modules; let $\aS[\Ass]$ be the
category of associative ring spectra (in $\aS$); and let
$\Ho\aS[\Ass]$ denote its homotopy category.  Let
$\Ho\aS[\Ass]^{\simeq}$ denote the subcategory of $\Ho\aS[\Ass]$ where
the maps are isomorphisms. There is a contravariant functor
$\THC$ from $\Ho\aS[\Ass]^{\simeq}$ to the homotopy category of
$E_{2}$ ring spectra together with canonical isomorphisms $\THC(R)\to
\CC(R)$ for those $R$ whose underlying objects of $\aS$ are
fibrant and cofibrant relative to the unit.
\end{main}

The correct generality for $\THC$ is the setting of small spectral
categories, which generalize associative ring spectra.  In
Section~\ref{sec:MS}, we explain that the McClure-Smith theory extends
to construct an $E_{2}$ structure on the $\Tot$ of the topological
Hochshild-Mitchell cosimplicial construction $\CC(\aC)$ for a small spectral
category $\aC$.  The natural weak equivalences for spectral categories
are the \term{Dwyer-Kan} equivalences, or DK-equivalences.  A spectral
functor $\phi\colon \aD\to \aC$ is a \term{DK-embedding} if it induces a
weak equivalence $\aD(a,b)\to \aC(\phi(a),\phi(b))$ for all objects $a,b$ of
$\aD$; a DK-equivalence is a DK-embedding that induces an equivalence
of \term{homotopy categories} $\pi_{0}\aD\to \pi_{0}\aC$.  

The following theorem, proved in Section~\ref{sec:pfdkfunct}, is the
natural generalization of Theorem~\ref{main:wefunct} to this setting;
the theorem roughly says that $\THC$ is functorial in DK-embeddings.  In it,
we use the condition for small spectral categories analogous to the
condition we used for associative ring spectra in
Theorem~\ref{main:wefunct}: We say that a small spectral category
$\aC$ is pointwise relatively cofibrant if the mapping spectra
$\aC(c,c)$ are cofibrant relative to the unit for all objects $c$ in
$\aC$ and the mapping spectra $\aC(c,d)$ are cofibrant for all pairs
of objects $c\neq d$ in $\aC$.  Similarly, we say a small spectral
category is pointwise fibrant if each mapping spectrum $\aC(c,d)$ is
fibrant.

\begin{main}\label{main:dkfunct}
Let $\aS\Cat$ denote the category of small spectral categories and\break
$\Ho(\aS\Cat)$ the category obtained by formally inverting the
DK-equivalences.  Let $\Ho(\aS\Cat)^{DK}$ be the subcategory of
$\Ho(\aS\Cat)$ generated by the DK-embeddings.   There is a contravariant
functor
$\THC$ from $\Ho(\aS\Cat)^{DK}$ to the homotopy category of
$E_{2}$ ring spectra together with canonical isomorphisms $\THC(\aC)\to
\CC(\aC)$ for those $\aC$ which are pointwise relatively cofibrant and
pointwise fibrant.
\end{main}

For any small spectral category $\aC$, we can construct a functorial
``thick closure'' $\perf(\aC)$~\cite[\S 5]{BM-tc} (after fixing a
cardinal bound); roughly speaking, this is
the full subcategory of spectral presheaves on $\aC$ generated under
finite homotopy colimits and retracts by $\aC$.  A spectral functor
$\phi \colon \aD\to \aC$ is a Morita equivalence when the induced
functor $\perf(\aC) \to \perf(\aD)$ is a DK-equivalence.  One reason
for interest in the Morita equivalences is that the Bousfield
localization of the category of small spectral categories at the
Morita equivalences is a model for the $\infty$-category of small
stable idempotent-complete $\infty$-categories~\cite[4.23]{BGT}.  The
following theorem, proved in Section~\ref{sec:pfme}, shows that $\THC$
descends to a functor on a subcategory of this localization.

\begin{main}\label{main:me}
If $\phi\colon \aD\to \aC$ is a Morita equivalence, then $\THC(\phi)\colon
\THC(\aC)\to \THC(\aD)$ is an isomorphism in the homotopy category of
$E_{2}$ ring spectra.
\end{main}

Theorems~\ref{main:wefunct} and~\ref{main:me} describe invariance
properties of $\THC$ analogous to the well-established invariance
properties of $THH$.  However, $\THC$ in 
fact has more general invariance properties.
For example, if $\aD$ is a small spectral subcategory of the category
of cofibrant-fibrant right $\aC$-modules that factors the Yoneda
embedding, then $\THC(\aC)\to \THC(\aD)$ is an isomorphism in the
homotopy category of $E_{2}$ ring spectra (see Example~\ref{ex:perf}
below).  The most general 
expression of this invariance we know can be expressed in terms of the
double centralizer condition, which is also a generalization of
derived Koszul duality.

Let $\aC$ and $\aD$ be small spectral categories and let $\aM$ be a
$(\aC,\aD)$-bimodule (a commuting left $\aC$-module and right
$\aD$-module structure; see Definition~\ref{defn:bibimod}
or~\ref{defn:sbimod} below). Then there are 
canonical maps in the categories of homotopical $(\aC,\aC)$-bimodules
and homotopical $(\aD,\aD)$-bimodules, respectively, 
\[
\aC\to \dR\Hom_{\aD^{\op}}(\aM,\aM),\qquad 
\aD\to \dR\Hom_{\aC}(\aM,\aM).
\]
A standard definition is that
$\aM$ satisfies the \term{double centralizer condition} when both
these maps are weak equivalences.  Working backward from this
terminology, we say that $\aM$ satisfies the \term{single centralizer
condition for $\aC$} when the first map (out of $\aC$) is a weak
equivalence and the \term{single centralizer
condition for $\aD$} when the second map (out of $\aD$) is a weak
equivalence.  The following is the spectral
version of the main theorem of Keller~\cite{Keller-DIH}; we prove it
in Section~\ref{sec:pfdcc}.

\begin{main}\label{main:dcc}
Let $\aC$ and $\aD$ be small spectral categories and $\aM$ a
$(\aC,\aD)$-bimodule that satisfies the single centralizer condition
for $\aD$; then there is a canonical map in the homotopy category of
$E_{2}$ ring spectra $\THC(\aC)\to\THC(\aD)$.  If $\aM$ satisfies
the double centralizer condition, then
$\THC(\aC)\to\THC(\aD)$ is an isomorphism in the homotopy category of
$E_{2}$ ring spectra.
\end{main}

We deduce Theorem~\ref{main:string} in Section~\ref{sec:pfstring} as
an immediate corollary of the previous theorem.  
Dwyer-Greenlees-Iyengar~\cite[4.22]{DwyerGreenleesIyengar-Duality}
relates the double centralizer condition for the sphere spectrum $\bS$
as a $(\splus \Omega X,DX)$-bimodule to the Eilenberg-Moore spectral sequence;
Section~3 of~\cite{BM-Koszul} describes nice models of $DX$ and
$\splus \Omega X$ and explicitly proves the double centralizer
condition when $X$ is a simply connected finite cell
complex for a bimodule whose underlying spectrum is equivalent to
$\bS$. 

Theorems~\ref{main:dkfunct} and~\ref{main:me} have an
$\infty$-categorical extension.  Specifically, we prove the following
theorem in Section~\ref{sec:inftyfunct}.

\begin{main}\label{main:inf}
Let $\Cat^{\ex}$ denote the $\infty$-category of small stable
idempotent-complete $\infty$-categories and exact functors.  Then
$\THC$ extends to a functor from the
subcategory of $\Cat^{\ex}$ where the morphisms are fully-faithful
inclusions to the $\infty$-category of $E_2$ ring spectra.
\end{main}

\subsection*{Conventions}

In this paper, $\aS$ denotes either the category of symmetric spectra
(of topological spaces), or the category of orthogonal spectra, or the
category of EKMM $S$-modules.  For brevity we call $\aS$ the
\term{category of spectra} and objects of $\aS$ \term{spectra}.  We
regard the \term{stable category} as the homotopy category obtained
from $\aS$ be formally inverting the weak equivalences.  (The words
``spectrum'' and ``spectra'' when used in this paper should not be
construed as refering to any other notion or category.)

\subsection*{Acknowledgments}
The authors thank Ralph Cohen, Zachery Lindsey, and Constantin
Teleman for useful conversations; we additionally thank Ralph Cohen
for suggesting this project.

\section{Bi-indexed spectra and the tensor-Hom adjunctions}
\label{sec:bis}

The purpose of this section is to establish technical foundations for
proving tensor-$\Hom$ adjunctions for modules over small spectral
categories.  To do this, we work here with the theory of
``bi-indexed'' spectra, which are like spectrally enriched directed
graphs but where the source and target vertices can 
be in different sets.  

\begin{defn}\label{defn:bis}
For sets $A$ and $B$, an $(A,B)$-spectrum $\aX$ consists of a choice of
spectrum $\aX(a,b)$ for each object $(a,b)$ of $A\times B$; we call
$(a,b)$ a \term{bi-index}.  A morphism
of $(A,B)$-spectra $\aX\to \aY$ consists of a map of spectra $\aX(a,b)\to
Y(a,b)$ for all bi-indexes $(a,b)\in A\times B$.  A bi-indexed spectrum $\aX$ is an
$(A,B)$-spectrum for some $A,B$; we define the \term{source} of $\aX$
(denoted $S(\aX)$) to
be $B$ and the \term{target} of $\aX$ to be $A$ (denoted $T(\aX)$).  If
$S(X)=S(Y)$ and $T(X)=T(Y)$, then the set of maps of bi-indexed spectra from $\aX$
to $\aY$ is the set of maps of $(T(\aX),S(\aX))$-spectra from $\aX$ to $\aY$;
otherwise, it is empty.  

For a bi-indexed spectrum $\aX$, let $\aX^{\op}$ denote the bi-indexed spectrum with
\begin{itemize}
\item $S(\aX^{\op})=T(\aX)$,
\item $T(\aX^{\op})=S(\aX)$, and
\item $\aX^{op}(s,t)=\aX(t,s)$ for all $(t,s)\in T(\aX)\times S(\aX)$.
\end{itemize}
\end{defn}

We have written and typically write generic bi-indexed spectra with
the target variable first and the source variable second; we refer to
this as the $TS$-indexing convention.  For the bi-indexed spectra
associated to small spectral categories (see Definition~\ref{defn:scat}
below), it is more usual to use the $ST$-indexing convention, writing
the source variable first and the target variable second, and we
follow this convention for spectral categories and their bimodules.
When it is unclear from the context which indexing is used, we add
a superscript $st$ or $ts$, so 
\[
\aX^{st}(a,b)=\aX^{ts}(b,a).
\]
We emphasize the distinction between $(-)^{st}$ and $(-)^{\op}$:
$(-)^{st}$ just reverses the notation of source and target, while
$(-)^{\op}$ reverses the notion of source and target.

As defined above, the category of bi-indexed spectra only admits maps
between objects whose source sets agree and target sets agree and so
it is sometimes useful to alter these sets.

\begin{defn}\label{defn:restriction}
Given functions $f\colon A'\to
A$, $g\colon B'\to B$ and an $(A,B)$-spectrum $\aX$, define the
\term{restriction of $\aX$ along $(f,g)$} to be the $(A',B')$-spectrum
$R_{f,g}\aX$ where
\[
(R_{f,g}\aX)(a',b')=\aX(f(a'),g(b'))
\]
for all $(a',b')\in A'\times B'$.  We define the \term{target
restriction of $\aX$ along $f$} and the \term{source
restriction of $\aX$ along $g$} to be the $(A',B)$-spectrum $T_{f}\aX$
and $(A,B')$-spectrum $S_{g}\aX$ where
\[
(T_{f}\aX)(a',b)=\aX(f(a'),b),\qquad (S_{g}\aX)(a,b')=\aX(a,g(b'))
\]
for all $(a',b)\in A'\times B$, $(a,b')\in A\times B'$.
\end{defn}

Since bi-indexed spectra are determined by their constituent spectra
on each bi-index, we have
\[
T_{g}S_{f} = R_{f,g} = S_{f}T_{g}
\]
and for $f'\colon A''\to A'$ and $g'\colon B''\to B'$,
\[
S_{f'}S_{f} = S_{f\circ f'}, \qquad 
R_{(f',g')}R_{f,g} = R_{f\circ f',g\circ g'},
\qquad T_{g'}T_{g} = T_{g\circ g'}.
\]
We could use the preceding definition to define a more sophisticated
category of bi-indexed sets incorporating non-identity maps on source
and target sets, but the advantage of the current approach is that
this category of bi-indexed spectra has a partial monoidal structure,
constructed as follows.

\begin{cons}\label{cons:tensor}
For $\aX$ an $(A,B)$-spectrum and $\aY$ a $(B,C)$-spectrum, define
$\aX\otimes \aY$ to be the $(A,C)$-spectrum
\[
(\aX\otimes \aY)(a,c) = \bigvee_{b\in B}\aX(a,b)\sma \aY(b,c).
\]
For $\aZ$ a $(C,D)$-spectrum, the associativity isomorphism for the
smash product and the universal property of coproduct induce an
associativity isomorphism 
\[
\alpha_{\aX,\aY,\aZ}\colon (\aX\otimes \aY)\otimes \aZ\iso \aX\otimes (\aY\otimes \aZ).
\]
For a set $A$, let $\bS_{A}$ be the $(A,A)$-spectrum where
\[
\bS(a_{1},a_{2})=\begin{cases}
*&a_{1}\neq a_{2}\\
\bS&a_{1}=a_{2}.
\end{cases}
\]
The left and right unit isomorphism for the smash product induce left
and right unit isomorphisms
\[
\eta^{\ell}_{\aX}\colon \bS_{A}\otimes \aX\iso \aX\qquad \text{and}\qquad 
\eta^{r}_{\aX}\colon \aX\otimes \bS_{B}\iso \aX.
\]
\end{cons}

The coherence of associativity and unit isomorphisms for spectra then
imply the following proposition.

\begin{prop}\label{prop:pmc}
The category of bi-indexed spectra is a partial monoidal category
under $\otimes$:  The object $\aX\otimes \aY$ is defined when $T(\aY)=S(\aX)$,
and whenever defined, the following associativity
\[
\xymatrix@C-5.25pc{%
&&(\aW\otimes \aX)\otimes (\aY\otimes \aZ)
\ar[drr]^-{\quad\alpha_{\aW,\aX,\aY\otimes \aZ}}\\
((\aW\otimes \aX)\otimes \aY)\otimes \aZ
\ar[urr]^-{\alpha_{\aW\otimes \aX,\aY,\aZ}\quad}
\ar[dr]_-{\alpha_{\aW,\aX,\aY}\otimes \id_{\aZ}\quad}
&&&&
\aW\otimes (\aX\otimes (\aY\otimes \aZ))\\
&(\aW\otimes (\aX\otimes \aY))\otimes \aZ
\ar[rr]_-{\alpha_{\aW,\aX\otimes \aY,\aZ}
\vrule height1em width0pt depth0pt}
&&
\aW\otimes ((\aX\otimes \aY)\otimes \aZ)
\ar[ur]_-{\quad\id_{\aW}\otimes \alpha_{\aX,\aY,\aZ}}
}
\]
and unit
\[
\xymatrix@C-1.25pc{%
(\aX\otimes \bS_{B})\otimes \aY
\ar[rr]^{\alpha_{\aX,\bS_{B},\aY}}
\ar[dr]_{\eta^{r}_{\aX}}&&
\aX\otimes (\bS_{B}\otimes \aY)
\ar[dl]^{\eta^{\ell}_{\aY}}&&
\\
&\aX\otimes \aY
}
\]
diagrams commute.
\end{prop}

The tensor product has two partially defined right adjoints and we
also construct a third closely related functor.

\begin{cons}\label{cons:homlrb}
If $\aX$
is an $(A,B)$-spectrum, $\aY$ is an $(A,C)$-spectrum, and $\aZ$ is a
$(D,B)$-spectrum, we define the $(B,C)$-spectrum $\Hom^{\ell}(\aX,\aY)$ as
\[
(\Hom^{\ell}(\aX,\aY))(b,c)=\prod_{a\in A} F(\aX(a,b),\aY(a,c))
\]
and the $(D,A)$-spectrum $\Hom^{r}(\aX,\aZ)$ as
\[
(\Hom^{r}(\aX,\aZ))(d,a)=\prod_{b\in B} F(\aX(a,b),\aZ(d,b))
\]
(where $F$ denotes the function spectrum construction, adjoint to the
smash product).  For $\aX'$ an $(A,B)$-spectrum, we define the
spectrum $\Hom^{b}(\aX,\aX')$ as
\[
\Hom^{b}(\aX,\aX')=\prod_{(a,b)\in A\times B}F(\aX(a,b),\aX'(a,b)).
\]
\end{cons}

We note that $\Hom^{b}$ provides a \term{partial spectral enrichment} of bi-indexed
spectra: when $\Hom^{b}(\aX,\aX')$ is defined, maps
of spectra from $\bS$ into $\Hom^{b}(\aX,\aX')$ are canonically in 
one-to-one correspondence with maps of bi-indexed spectra from $\aX$ to $\aX'$,
and when $\Hom^{b}(\aX,\aX')$ is not defined, the set of maps of
bi-indexed spectra from $\aX$ to $\aX'$ is empty.
An easy check of definitions shows the following adjunction property.

\begin{prop}\label{prop:bisadj}
Let $\aX$ be an $(A,B)$-spectrum, $\aY$ a $(B,C)$-spectrum, and $\aZ$
an $(A,C)$-spectrum.  Then there are canonical isomorphisms of spectra
\[
\Hom^{b}(\aX,\Hom^{r}(\aY,\aZ))\iso
\Hom^{b}(\aX\otimes \aY,\aZ)\iso
\Hom^{b}(\aY,\Hom^{\ell}(\aX,\aZ)).
\]
\end{prop}

When $S(\aX)=T(\aX)=O$ for some set $O$, $\aX$ is precisely a
small spectral $O$-graph (with the reverse convention on the order of
variables, i.e., with the $TS$-indexing convention); the tensor product above restricts to a monoidal 
product on $O$-graphs and it is well known that
the category of small spectral $O$-categories is isomorphic to the
category of monoids for this monoidal product
(see~\cite[\S6.2]{SSMonoidalEq}; compare~\cite[\S
II.7]{MacLane-Categories}).  We say more about this below in
Section~\ref{sec:rewrite}.  Partly to avoid confusion with the
indexing conventions, we will call the monoids under this convention bi-indexed
ring spectra.

\begin{defn}\label{defn:birs}
A \term{bi-indexed ring spectrum} is a monoid for $\otimes$ in
bi-indexed spectra.  For a bi-indexed ring spectrum $\aX$,
the \term{object set} $O(\aX)$ is $S(\aX)=T(\aX)$.
\end{defn}

Note that with the above definition, the natural morphisms for
bi-indexed ring spectra only allow maps between small spectral categories
with the same object sets.  Instead of defining the analogue of
spectral functors directly, it is more convenient to work with bimodules.

\begin{defn}\label{defn:bibimod}
Let $\aX$ and $\aY$ be bi-indexed ring spectra.  An
$(\aX,\aY)$-bimodule consists of a bi-indexed spectrum $\aM$ together
with a left $\aX$-object structure (for $\otimes$) and a commuting
right $\aY$-object structure.  We write $\Mod_{\aX,\aY}$ for the category of
$(\aX,\aY)$-bimodules. 
\end{defn}

Commuting here means that for the left-object structure $\xi\colon
\aX\otimes \aM\to \aM$ and the right object structure $\upsilon\colon
\aM\otimes \aY\to \aM$, the diagram
\[
\xymatrix@C-1pc{%
(\aX\otimes \aM)\otimes \aY\ar[rr]^{\alpha_{\aX,\aM,\aY}}
\ar[d]_{\xi\otimes \id_{\aY}}
&&\aX\otimes (\aM\otimes \aY)
\ar[d]^{\id_{\aX}\otimes \upsilon}\\
\aM\otimes \aY\ar[dr]_{\upsilon}&&\aX\otimes \aM\ar[dl]^{\xi}\\
&\aM
}
\]
commutes.  The left and right object structures 
require (and are defined by the requirement that) the associativity
\[
\xymatrix@R-1pc@C+.75pc{%
(\aX\otimes \aX)\otimes \aM\ar[r]^-{\alpha_{\aX,\aX,\aM}}\ar[d]_-{\mu_{\aX}\otimes \id_{\aM}}
&\aX\otimes (\aX\otimes \aM)\ar[r]^-{\id_{\aX}\otimes \xi}
&\aX\otimes \aM\ar[d]^-{\xi}\\
\aX\otimes \aM\ar[rr]_-{\xi}&&\aM\\
(\aM\otimes \aY)\otimes \aY\ar[r]^-{\alpha_{\aM,\aY,\aY}}\ar[d]_-{\upsilon\otimes \id_{\aY}}
&\aM\otimes (\aY\otimes \aY)\ar[r]^-{\id_{\aM}\mu_{\aY}}
&\aM\otimes \aY\ar[d]^-{\upsilon}\\
\aM\otimes \aY\ar[rr]_-{\upsilon}&&\aM
}
\]
and unit
\[
\xymatrix@R-1pc@C+.75pc{%
\bS_{O(\aX)}\otimes \aM\ar[r]^-{\eta_{\aX}\otimes \id_{\aM}}\ar[dr]_-{\eta^{\ell}}
&\aX\otimes \aM\ar[d]^-{\xi}
&
\aM\otimes \bS_{O(\aY)}\ar[r]^-{\id_{\aM}\otimes \eta_{\aY}}\ar[dr]_-{\eta^{r}}
&\aM\otimes \aY\ar[d]^-{\upsilon}
\\
&\aM&&\aM
}
\]
diagrams commute, where $\mu_{\aX},\mu_{\aY}$ denote the multiplications
and $\eta_{\aX},\eta_{\aY}$ denote the units for the monoid structures
on $\aX$ and $\aY$.

Given a function $\phi \colon O(\aY)\to O(\aX)$, we obtain an
$(O(\aX),O(\aY))$-spectrum $S_{\phi}\aX=\aX(-,\phi(-))$, which has a canonical
left $\aX$-object structure, given by the monoid structure of $\aX$.
We explain in Section~\ref{sec:rewrite} why the following definition
captures the correct notion of spectral functor.

\begin{defn}\label{defn:bifunc}
Let $\aX$ and $\aY$ be bi-indexed ring spectra.  A \term{spectral
functor} $\phi\colon \aY\to \aX$ consists of a function
$\phi\colon O(\aY)\to O(\aX)$ and an $(\aX,\aY)$-bimodule structure on
the left $\aX$-object $S_{\phi}\aX=\aX(-,\phi(-))$.
\end{defn}

The Hochshild-Mitchell construction requires a version of $\Hom^{b}$
``over'' a pair of small spectral categories and the adjunction of
Proposition~\ref{prop:bisadj} suggests the utility of analogues of
$\otimes$, $\Hom^{\ell}$, and $\Hom^{r}$ over small spectral
categories.  

\begin{cons}
Let $\aX$, $\aY$, and $\aZ$ be bi-indexed ring spectra.
If $\aM$ is an $(\aX,\aY)$-bimodule and $\aN$ is a
$(\aY,\aZ)$-bimodule, then we define the $(\aX,\aZ)$-bimodule
$\aM\otimes_{\aY}\aN$ to be the usual coequalizer
\[
\xymatrix@C-1pc{%
\aM\otimes \aY\otimes \aN
\ar@<-.5ex>[r]\ar@<.5ex>[r]
&\aM\otimes \aN\ar[r]
&\aM\otimes_{\aY}\aN
}
\]
of the left and right $\aY$-actions.
If $\aM$ is an $(\aX,\aY)$-bimodule and $\aP$ is an $(\aX,\aZ)$-bimodule,
we define the $(\aY,\aZ)$-bimodule $\Hom^{\ell}_{\aX}(\aM,\aP)$ to be the usual
equalizer
\[
\xymatrix@C-1pc{%
\Hom^{\ell}_{\aX}(\aM,\aP)\ar[r]
&\Hom^{\ell}(\aM,\aP)\ar@<.5ex>[r]\ar@<-.5ex>[r]
&\Hom^{\ell}(\aX\otimes \aM,\aP)
}
\]
where one map is induced by the left $\aX$-action on $\aM$ and the
other map is composite of the left $\aX$-action on
$\aP$ and the map
\[
\Hom^{\ell}(\aM,\aP)
\to \Hom^{\ell}(\aX\otimes \aM,\aX\otimes \aP)
\]
adjoint to the map
\[
\aX\otimes \aM \otimes \Hom^{\ell}(\aM,\aP)\to
\aX\otimes \aP
\]
that applies the counit of the $\aM\otimes (-)$, $\Hom(\aM,-)$
adjunction. 
If $\aM$ is an $(\aX,\aY)$-bimodule and $\aQ$ is a $(\aZ,\aY)$-bimodule,
we define the $(\aZ,\aX)$-bimodule $\Hom^{r}_{\aY}(\aM,\aQ)$ to be the usual
equalizer
\[
\xymatrix@C-1pc{%
\Hom^{r}_{\aY}(\aM,\aQ)\ar[r]
&\Hom^{r}(\aM,\aQ)\ar@<.5ex>[r]\ar@<-.5ex>[r]
&\Hom^{r}(\aM\otimes \aY,\aQ)
}
\]
using the analogous pair of maps for $\Hom^{r}$.
If $\aM$ is an $(\aX,\aY)$-bimodule and $\aM'$ is an $(\aX,\aY)$-bimodule,
we define the spectrum $\Hom^{b}_{\aX,\aY}(\aM,\aM')$ to be the usual
equalizer
\[
\xymatrix@C-1pc{%
\Hom^{b}_{\aX,\aY}(\aM,\aM')\ar[r]
&\Hom^{b}(\aM,\aM')\ar@<.5ex>[r]\ar@<-.5ex>[r]
&\Hom^{b}(\aX\otimes \aM\otimes \aY,\aM')
}
\]
with the analogous pair of maps for $\Hom^{b}$.
\end{cons}

An easy check shows that the spectrum of maps $\Hom^{b}_{\aX,\aY}$
provides a spectral enrichment of the category of
$(\aX,\aY)$-bimodules. 

Proposition~\ref{prop:bisadj} now generalizes to the following
proposition.  The proof is again purely formal.

\begin{prop}\label{prop:bimadj}
Let $\aX$, $\aY$, and $\aZ$ be bi-indexed ring spectra.  Let $\aM$ be
an $(\aX,\aY)$-bimodule, let $\aN$ be a $(\aY,\aZ)$-bimodule, and let
$\aP$ be an $(\aX,\aZ)$ bimodule.  Then there are canonical
isomorphisms of spectra
\begin{align*}
&\hspace{-3em}\Hom^{b}_{\aX,\aZ}(\aM\otimes_{\aY} \aN,\aP)\hspace{-3em}\\
\Hom^{b}_{\aX,\aY}(\aM,\Hom^{r}_{\aZ}(\aN,\aP))\iso&&
\iso\Hom^{b}_{\aY,\aZ}(\aN,\Hom^{\ell}_{\aX}(\aM,\aP)).
\end{align*}
\end{prop}

We note that $\otimes$ and $\otimes_{\aY}$ commute with target
restriction (Definition~\ref{defn:restriction}) on the first variable
and source restriction on the second variable; $\Hom^{\ell}$ and
$\Hom^{\ell}_{\aX}$ convert source restriction on the first variable
to target restriction and preserves source restriction on the second
variable.  Likewise, $\Hom^{r}$ and $\Hom^{r}_{\aY}$ convert target
restriction on the first variable to source restriction and preserve
target restriction on the second variable.

The balanced tensor product produces the composition of spectral
functors for the definition of spectral functors
(Definition~\ref{defn:bifunc}) above.  Given a spectral functor $\phi$
from 
$\aY$ to $\aX$ and a spectral functor $\theta$ from $\aZ$ to $\aY$, using
the $(\aX,\aY)$-bimodule structure on $S_{\phi}\aX$ inherent in
$\phi$, we can make sense of 
the tensor product over $\aY$ on the right and construct a map of left
$\aX$-objects 
\[
S_{\phi}\aX\otimes_{\aY}S_{\theta}\aY \to S_{\phi\circ \theta}\aX.
\]
This map is an isomorphism because $\otimes_{\aY}$ commutes with
source restriction in the second variable; intrinsically, for every
fixed $x\in O(\aX)$ and $z\in O(\aZ)$, the diagram
\[
\xymatrix@C-1pc@R-2pc{%
\coprod\limits_{y_{0},y_{1}\in O(\aY)}\hspace{-2ex}
\aX(x,\phi(y_{0}))\sma \aY(y_{0},y_{1})\sma \aY(y_{1},\theta(z))
\ar@<-.5ex>[r]\ar@<.5ex>[r]
&\coprod\limits_{y\in O(\aY)}\hspace{-1ex}
\aX(x,\phi(y))\sma \aY(y,\theta (z))\\
&\longrightarrow\aX(x,\phi(\theta(z)))
}
\]
is a split coequalizer.  Using the isomorphism to give
$S_{\phi\circ\theta}\aX$ a right $\aZ$-action makes it an
$(\aX,\aZ)$-bimodule.  We define 
the composite of the spectral functors $\phi \circ \theta$ to consist of the
object function $\phi \circ \theta$ and this bimodule structure on $S_{\phi
\circ \theta}\aX$.

\section{Small spectral categories and the tensor-Hom adjunctions}
\label{sec:rewrite}

This section translates the work from the previous section to the
framework of small spectral categories.  When working in this
framework, we use the $ST$-indexing convention as this is standard in
this context.  We begin by reviewing the
definitions.

\begin{defn}\label{defn:scat}
A \term{small spectral category} is a small category enriched over
spectra. It consists of: 
\begin{enumerate}
\item a set of objects $O(\aC)$,
\item a spectrum $\aC(a,b)$ for each pair of objects $a,b\in O(\aC)$,
\item a unit map $\bS\to \aC(a,a)$ for each object $a \in O(\aC)$, and
\item a composition map $\aC(b,c)\sma \aC(a,b)\to \aC(a,c)$ for each
triple of objects $a,b,c\in O(\aC)$,
\end{enumerate}
satisfying the usual associativity and unit properties.  A
\term{strict morphism} $\aC\to \aC'$ of small spectral categories with the
same object set consists of a map of spectra $\aC(a,b)\to \aC'(a,b)$
for every pair of objects $a,b\in O(\aC)=O(\aC')$ that commutes with
the unit and composition maps; there are no strict morphisms between
small spectral categories with different object sets.
\end{defn}

We have inverse functors between the category of bi-indexed ring
spectra and small spectral categories (with strict morphisms) defined
as follows.  For a bi-indexed ring spectrum $\aX$, let $C_{\aX}$ be
the small spectral category defined by setting
\begin{enumerate}
\item the object set $O(C_{\aX})=O(\aX)$,
\item the mapping spectra $C_{\aX}(a,b)=\aX^{st}(a,b)=\aX^{ts}(b,a)$ for all $a,b\in O(C_{\aX})$,
\item the unit $\bS\to C_{\aX}(a,a)$ to be the map induced by the
monoid structure unit
$\bS_{O(\aX)}\to \aX$ for all $a\in O(C_{\aX})$, and
\item the composition $C_{\aX}(b,c)\sma C_{\aX}(a,b)\to C_{\aX}(a,c)$
to be the map $\aX(c,b)\sma \aX(b,a)\to \aX(c,a)$ that appears as a
wedge summand in the monoid structure multiplication $\aX\otimes \aX\to \aX$.
\end{enumerate}
Similarly, for a small spectral category $\aC$, we define a bi-indexed
ring spectrum $B_{\aC}$ with  the same object set by taking
$B_{\aC}(a,b)=\aC(b,a)$ and the obvious unit and multiplication.
These assignments are evidently functorial.

\begin{prop}\label{prop:iso}
The functors $C$ and $B$ above are inverse isomorphisms of categories
between the category of bi-indexed ring spectra and the category of
small spectral categories (with strict morphisms).
\end{prop}

The more usual category of small spectral categories has morphisms given by
\term{spectral functors}, which are simply the spectrally enriched
functors.  The following theorem relates this notion to
Definition~\ref{defn:bifunc}.  We prove it at the end of the section
after reviewing more of the theory of small spectral categories and their
modules.

\begin{thm}\label{thm:functor}
There is a canonical bijection
between the set of spectral functors of small spectral categories $\aD\to
\aC$ 
and the set of spectral functors of the corresponding bi-indexed ring spectra.
This bijection is compatible with composition.
\end{thm}

Left and right modules are basic notions for small spectral categories that
do not precisely correspond to left and right objects for bi-indexed
ring spectra.

\begin{defn}
Let $\aC$ be a small spectral category.  The (spectrally enriched)
category $\Mod_{\aC}$ of \term{left $\aC$-modules} is the (spectrally enriched) category
of spectrally enriched functors from $\aC$ to spectra; the (spectrally enriched)
category $\Mod_{\aC^{\op}}$ of \term{right $\aC$-modules} is the (spectrally enriched) category
of spectrally enriched contravariant functors from $\aC$ to spectra.
\end{defn}

For any one-point set $\{a\}$, the category of left
$\aC$-modules is isomorphic to the full subcategory category of left
$B_{\aC}$-objects with source set $\{a\}$ and is isomorphic as a
spectrally enriched 
category to the category of $(B_{\aC},\bS_{\{a\}})$-bimodules.  Likewise,
the category of right $\aC$-modules is isomorphic to the full
subcategory category of right $B_{\aC}$-objects with target set $\{a\}$
and is isomorphic as a spectrally enriched category to the category of
$(\bS_{\{a\}},B_{\aC})$-bimodules.  The category of left $B_{\aC}$-objects
is essentially the category of (singly) indexed left $B_{\aC}$-modules: a left
$B_{\aC}$-object $\aM$ consists of a left $\aC$-module $\aM^{st}(a,-)$
for each $a$ in $S(\aM)$.

Bimodules for small spectral categories do correspond precisely with
bimodules for bi-indexed ring spectra.  In the context of bimodules of
small spectral categories, just as in the context of bi-indexed
spectra, we take the convention that the category on the
left has the left action and the category on the right has the right
action. However, as always in the context of small spectral category
concepts, we follow the $ST$-indexing convention implicit in the definition
below that the righthand variable is the covariant one
while the lefthand variable is the contravariant one. 

\begin{defn}\label{defn:sbimod}
Let $\aC$ and $\aD$ be small spectral categories.  Let $\aD^{\op}\sma \aC$
be the small spectral category with objects $O(\aD^{\op}\sma
\aC)=O(\aD)\times O(\aC)$, mapping spectra 
\[
(\aD^{\op}\sma \aC)((d,c),(d',c'))=\aD(d',d)\sma \aC(c,c'),
\]
unit induced by the units of $\aC$ and $\aD$ (and the canonical
isomorphism $\bS\sma \bS\iso\bS$), and composition induced by the
composition on $\aD$  (performed
backwards) and the composition on $\aC$:
\begin{multline*}
(\aD^{\op}\sma \aC)((d',c'),(d'',c''))\sma (\aD^{\op}\sma \aC)((d,c),(d',c'))\\
=(\aD(d'',d')\sma \aC(c',c''))\sma (\aD(d',d)\sma \aC(c,c'))\\
\iso (\aD(d',d)\sma\aD(d'',d'))\sma (\aC(c',c'')\sma \aC(c,c'))\\
\to \aD(d'',d)\sma \aC(c,c'')
=(\aD^{\op}\sma \aC)((d,c),(d'',c'')).
\end{multline*}
The (spectrally enriched) category $\Mod_{\aC,\aD}$ of
\term{$(\aC,\aD)$-bimodules} is
the (spectrally enriched) category of
spectrally enriched functors from $\aD^{\op}\sma \aC$ to spectra.  
\end{defn}

Given a $(\aC,\aD)$-bimodule $\aF$, we write $B_{\aF}$ for the
$(O(\aC),O(\aD))$-spectrum 
\[
B_{\aF}(c,d)=B^{st}_{\aF}(d,c)=\aF(d,c)
\]
for $(c,d)\in
O(\aC)\times O(\aD)$.  This has a canonical
$(B_{\aC},B_{\aD})$-bimodule structure with action maps induced by
\[
B_{\aC}(c,c')\sma B_{\aF}(c',d)=
\aC(c',c)\sma \aF(d,c')\to \aF(d,c)=B_{\aF}(c,d)
\]
and
\[
B_{\aF}(c,d)\sma B_{\aD}(d,d')
=\aF(d,c)\sma \aD(d',d)\to \aF(d',c)=B_{\aF}(c,d').
\]
This is evidently functorial, and indeed extends canonically to a
spectrally enriched functor from the category of $(\aC,\aD)$-bimodules
to the category of $(B_{\aC},B_{\aD})$-bimodules (in bi-indexed spectra).

\begin{prop}
The spectrally enriched functor $B$ from $(\aC,\aD)$-bimodules to
$(B_{\aC},B_{\aD})$-bimodules (in bi-indexed spectra) is an
isomorphism of spectrally enriched categories.
\end{prop}

We write the inverse isomorphism as $C$; evidently,
$C_{\aM}(c,d)=\aM(d,c)$ for all $(c,d)\in O(\aC)\times O(\aD)$.

In light of the previous proposition, for $(\aC,\aD)$-bimodules $\aF$
and $\aG$, we write $\Hom^{b}_{\aC,\aD}(\aF,\aG)$ for the spectrum of
bimodule maps from $\aF$ to $\aG$ and we more generally define $\otimes$, $\otimes_{\aD}$,
$\Hom^{\ell}$, $\Hom^{\ell}_{\aC}$, $\Hom^{r}$, and $\Hom^{r}_{\aD}$ in terms of the inverse
isomorphisms $B$ and $C$ (for spectral categories / bi-indexed ring spectra and
bimodules). In explicit terms, we have:

\begin{prop}\label{prop:formulas}
Let $\aA$, $\aB$, and $\aC$ be small spectral categories.
\begin{enumerate}
\item\label{i:tensor} For an $(\aA,\aB)$-bimodule $\aF$ and a $(\aB,\aC)$-bimodule
$\aG$, the $(\aA,\aC)$-bimodule $\aF\otimes \aG=C_{B_{\aF}\otimes
B_{\aG}}$ satisfies
\[
(\aF\otimes \aG)(c,a)=\bigvee_{b\in O(\aB)}\aF(b,a)\sma \aG(c,b)
\]
and $\aF\otimes_{\aB} \aG=C_{B_{\aF}\otimes_{B_{\aB}}
B_{\aG}}$ is the coequalizer
\[
\xymatrix@C-1pc{%
\aF\otimes \aB\otimes \aG
\ar@<.5ex>[r]\ar@<-.5ex>[r]
&\aF\otimes \aG\ar[r]
&\aF\otimes_{\aB}\aG.
}
\]
\item\label{i:Homl} For an $(\aA,\aB)$-bimodule $\aF$ and an $(\aA,\aC)$-bimodule
$\aG$, the $(\aB,\aC)$-bimodule
$\Hom^{\ell}(\aF,\aG)=C_{\Hom^{\ell}(B_{\aF},B_{\aG})}$ satisfies
\[
(\Hom^{\ell}(\aF,\aG))(c,b)=\prod_{a\in O(\aA)}F(\aF(b,a),\aG(c,a))
\]
and
$\Hom^{\ell}_{\aA}(\aF,\aG)=C_{\Hom^{\ell}_{B_{\aA}}(B_{\aF},B_{\aG})}$
is the equalizer
\[
\xymatrix@C-1pc{%
\Hom^{\ell}_{\aA}(\aF,\aG)\ar[r]
&\Hom^{\ell}(\aF,\aG)
\ar@<.5ex>[r]\ar@<-.5ex>[r]
&\Hom^{\ell}(\aA\otimes \aF,\aG).
}
\]
\item\label{i:Homr} For an $(\aA,\aB)$-bimodule $\aF$ and a $(\aC,\aB)$-bimodule
$\aG$, the $(\aC,\aA)$-bimodule
$\Hom^{r}(\aF,\aG)=C_{\Hom^{r}(B_{\aF},B_{\aG})}$ satisfies
\[
(\Hom^{r}(\aF,\aG))(a,c)=\prod_{b\in O(\aB)}F(\aF(b,a),\aG(b,c))
\]
and
$\Hom^{r}_{\aB}(\aF,\aG)=C_{\Hom^{r}_{B_{\aB}}(B_{\aF},B_{\aG})}$
is the equalizer
\[
\xymatrix@C-1pc{%
\Hom^{r}_{\aB}(\aF,\aG)\ar[r]
&\Hom^{r}(\aF,\aG)
\ar@<.5ex>[r]\ar@<-.5ex>[r]
&\Hom^{r}(\aF\otimes \aB,\aG).
}
\]
\end{enumerate}
\end{prop}

As an immediate consequence of Proposition~\ref{prop:bimadj}, we
obtain the corresponding adjunction in the context of small spectral categories.

\begin{prop}\label{prop:scbimadj}
Let $\aA$, $\aB$, and $\aC$ be small spectral categories.  Let $\aF$ be
an $(\aA,\aB)$-bimodule, $\aG$ be a $(\aB,\aC)$-bimodule, and let
$\aH$ be an $(\aA,\aC)$ bimodule.  Then there are canonical
isomorphisms of spectra
\begin{align*}
&\hspace{-3em}\Hom^{b}_{\aA,\aC}(\aF\otimes_{\aB} \aG,\aH)\hspace{-3em}\\
\Hom^{b}_{\aA,\aB}(\aF,\Hom^{r}_{\aC}(\aG,\aH))\iso&&
\iso\Hom^{b}_{\aB,\aC}(\aG,\Hom^{\ell}_{\aA}(\aF,\aH)).
\end{align*}
\end{prop}

Comparing the formulas in Proposition~\ref{prop:formulas} with the
intrinsic definition of the spectral enrichment of a category of
spectral functors reveals the following relationship between
$\Hom^{\ell}_{\aC}$ and the spectral enrichment on the category of
left $\aC$-modules, which is essentially a special case
of the observation on $\Hom^{\ell}_{\aX}$ and source restriction in
the previous section.  An analogous result holds for
$\Hom^{r}_{\aC}$ and the spectral enrichment on the category of right
$\aC$-modules.

\begin{prop}\label{prop:homover}
In the notation of Proposition~\ref{prop:formulas}.(\ref{i:Homl}), 
\[
(\Hom^{\ell}_{\aA}(\aF,\aG))^{st}(c,b)
\iso \Mod_{\aA}(\aF(b,-),\aG(c,-))
\]
for all $b\in O(\aB)$, $c\in O(\aC)$.
\end{prop}

Finally, we return to Theorem~\ref{thm:functor}.

\begin{proof}[Proof of Theorem~\ref{thm:functor}]
Let $\aX$ and $\aY$ be bi-indexed ring spectra and let $\aC$ and $\aD$
denote the corresponding small spectral categories.

Given a spectral functor $\psi\colon \aD\to \aC$, let $\aF_{\psi}$
denote the $(\aC,\aD)$-bimodule with component spectra
$\aF^{st}_{\psi}(d,c)=\aC(\psi(d),c)$ and the evident bimodule structure.
Then the underlying left $\aX$-object of the $(\aX,\aY)$-bimodule
$B_{\aF_{\psi}}$ is $S_{\psi}\aX$.
Let $B_{\psi}$ be the spectral functor $\aX \to \aY$ that uses the
underlying object function of $\psi$ as the function on object sets and
$B_{\aF_{\psi}}$ as specifying the bimodule structure on
$S_{\psi}\aX$.

Given a spectral functor 
$\phi\colon \aY \to \aX$, we obtain a spectral functor $C_{\phi}$
from $\aD$ to $\aC$ using the same object function and
the map on morphism spectra defined as follows.
The map of left $\aX$-objects $S_{\phi}\aX\otimes \aY\to S_{\phi}\aX$ is 
adjoint to a map of bi-indexed spectra $\aY\to
\Hom^{\ell}_{\aX}(S_{\phi}\aX,S_{\phi}\aX)$.  
Because $\Hom^{\ell}$ converts source restriction in the first
variable and preserves source restriction in the second variable, we
have a canonical isomorphism 
\[
\Hom^{\ell}_{\aX}(S_{\phi}\aX,S_{\phi}\aX)\iso
R_{\phi,\phi}\Hom^{\ell}_{\aX}(\aX,\aX)
\iso R_{\phi,\phi}\aX.
\]
The map of bi-indexed spectra $\aY\to R_{\phi,\phi}\aX$ then specifies
a map $\aD(a,b)\to \aC(\phi(a),\phi(b))$ for all $a,b\in O(\aD)$.
In light of Proposition~\ref{prop:homover}, this map is the composite
\[
\aD(a,b)\to F_{\aC}(\aC(\phi(b),-),\aC(\phi(a),-))\iso 
\aC(\phi(a),\phi(b))
\]
of the adjoint of $\aC(\phi(b),-)\sma \aD(a,b)\to \aC(\phi(a),-)$ and
the enriched Yoneda lemma isomorphism.  From here it follows easily
that the constructed map on morphism spectra preserves units and
composition. 

It is clear that $B_{C_{\phi}}=\phi$, $C_{B_{\psi}}=\psi$, and
moreover that $B$ preserves composition of spectral functors.
\end{proof}

Using the enriched form of the Yoneda lemma, it is straightforward to
check that natural transformations of spectral functors between small
spectral categories correspond to maps of bimodules for spectral
functors between bi-indexed ring spectra; we do not use this result.

\section{Hochschild-Mitchell and McClure-Smith constructions}\label{sec:MS}

In this section, we review the point-set construction of topological
Hochschild cohomology of a small spectral category in terms of the
Hochschild-Mitchell complex $\CC$.  We then observe that this fits into the
framework of the McClure-Smith approach to the Deligne conjecture; in
particular, there is a natural $E_{2}$ ring spectrum structure on
$\CC$.

\begin{cons}[Topological Hochschild-Mitchell construction]\label{cons:HM}
Let $\aC$ be a\break small spectral category and $\aM$ a $(\aC,\aC)$-bimodule.
Let $\CC\supdot(\aC;\aM)$ be the cosimplicial spectrum
$\Hom^{b}_{\aC,\aC}(B\subdot(\aC;\aC;\aC),\aM)$, where $B\subdot$
denotes the two-sided bar construction for the monoidal product
$\otimes$.  More concretely, $\CC\supdot(\aC;\aM)$ is the cosimplicial
spectrum  which in cosimplicial degree $n$ is
\[
\Hom^{b}_{\aC,\aC}(\aC\otimes 
\underbrace{\aC\otimes \cdots \otimes \aC}_{n\text{ factors}}
\otimes \aC,\aM)
\]
with coface map $\delta^{i}$ induced by the product $\aC\otimes \aC\to
\aC$ (q.v.~Proposition~\ref{prop:iso}, Definition~\ref{defn:birs}) on
the $i$th, $(i+1)$th factors (starting the count from zero outside the
braces) and codegeneracy $\sigma^{i}$ maps the unit map
$C_{\bS_{O(\aC)}}\to \aC$ (q.v.~Proposition~\ref{prop:pmc}) inserting
the $\aC$ as the $i$th factor.  We write $\CC\supdot(\aC)$ for
$\CC\supdot(\aC;\aC)$ in the case $\aM = \aC$.  Let $CC(\aC;\aM)$ and
$CC(\aC)$ denote the spectra obtained by applying $\Tot$.
\end{cons}

Construction~\ref{cons:HM} is evidently covariantly functorial in maps
of the 
bimodule $\aM$ and contravariantly functorial in spectral functors of
the small spectral category $\aC$ (pulling back the bimodule structure along
the spectral functor).  Without hypotheses on $\aC$ and $\aM$, the
topological Hochschild-Mitchell construction may not preserve weak
equivalences.  However, when $\aC$ is pointwise relatively cofibrant
(see Definition~\ref{defn:relcof}) and $\aM$ is pointwise fibrant, $\CC$
preserves weak equivalences in each variable; see
Proposition~\ref{prop:hypreedy} and Theorem~\ref{thm:inv}. 

The free, forgetful adjunction arising from the interpretation of small
spectral categories as monoids for $\otimes$
(Proposition~\ref{prop:iso}) allows us to rewrite the
cosimplicial object in Construction~\ref{cons:HM} more explicitly as
\[
\CC^{0}(\aC,\aM)=\prod_{c}\aM(c,c)
\]
and
\begin{equation}\label{eq:HM}
\begin{aligned}
\CC^{n}(\aC;\aM)&\iso\Hom^{b}(\underbrace{\aC\otimes \cdots \otimes \aC}_{n\text{ factors}},\aM)\\
&\iso \prod_{c_{0},\ldots,c_{n}}F(\aC(c_{1},c_{0})\sma \cdots \sma\aC(c_{n},c_{n-1}),\aM(c_{n},c_{0})). 
\end{aligned}
\end{equation}
In this form, the faces $\delta^{1},\ldots,\delta^{n-1}\colon
\CC^{n-1}\to \CC^{n}$ are induced by the composition in the category
with $\delta^{0},\delta^{n}$ induced by the bimodule structure on
$\aM$; the degeneracies $\sigma^{i}\colon \CC^{n}\to \CC^{n+1}$ are
induced by $\bS\to \aC(c_{i},c_{i})$, inserting the identity map in
the $i$th position.  This is the usual explicit description of the
Hochschild-Mitchell construction for an enriched category.

Thinking in terms of the partial monoidal category of bi-indexed
spectra (Section~\ref{sec:bis}), the description of $\CC\supdot(\aC)$
as $\Hom^{b}(\aC\otimes \cdots \otimes \aC,\aC)$ identifies
$\CC\supdot(\aC)$ as the (non-symmetric) endomorphism operad
$\oEndb(B_{\aC})$ of the 
corresponding bi-indexed spectrum $B_{\aC}$.  Because $B_{\aC}$ is a 
monoid for the monoidal product, there is a canonical map $\bS\to
\oEndb(B_{\aC})(n)$ for all $n$ induced by the iterated multiplication
\[
B_{\aC}\otimes \cdots \otimes B_{\aC}\to B_{\aC}.
\]
These assemble to a map of non-symmetric operads $\Ass\to
\oEndb(B_{\aC})$, where $\Ass$ denotes the non-symmetric associative
operad in spectra $\Ass(n)=\bS$.  This is precisely a ``operad with
multiplication'' in the terminology of 
McClure-Smith~\cite[10.1]{McClureSmith-CosimplicialCubes}.  The point
of identifying this structure is that it is the data required in the
McClure-Smith theory to
induce an $E_{2}$ ring structure on $\Tot$.  Specifically, as a
consequence of~\cite[9.1,10.3]{McClureSmith-CosimplicialCubes}, we can
immediately deduce the following proposition.

\begin{prop}
The topological Hochschild-Mitchell construction $\CC(\aC))$ has a
canonical structure of a $\oD_{2}$-algebra ($E_{2}$ ring spectrum)
where $\aD_{2}$ is the $E_{2}$ operad of
McClure-Smith~\cite[\S9]{McClureSmith-CosimplicialCubes}. 
\end{prop}

\begin{proof}
In terms of the maps $e\colon \bS\to \oEndb(B_{\aC})(0)$ and
$\mu\colon \bS\to \oEndb(B_{\aC})(2)$, under the isomorphism of
$\oEndb(B_{\aC})(n)$ with $\CC^{n}(\aC)$, the face and degeneracy maps
for $\CC(\aC)$ above coincide with the ones
described on p.~1136 of~\cite{McClureSmith-CosimplicialCubes} in the
proof of Theorem~10.3: for $f\in \oEndb(B_{\aC})$
\[
\sigma^{i}f=f\circ_{i+1}e,
\qquad \qquad 
\delta^{i}f=\begin{cases}
\mu\circ_{2}f&i=0\\
f\circ_{i}\mu&i=1,\ldots,n\\
\mu \circ_{1}f&i=n+1. 
\end{cases}
\qedhere
\]
\end{proof}

There is no naturality statement in the preceding proposition because
$\CC(\aC)$ is not functorial in $\aC$ (on the point-set level) in any
reasonable way.  It does have a very limited functoriality for
spectral functors that induce isomorphisms on mapping spectra; we
refer to these as \term{strictly fully faithful} spectral functors.
For a strictly fully faithful spectral functor $\phi \colon \aD\to
\aC$, there is a restriction map $\phi^{*}\colon \oEndb(\aC)\to
\oEndb(\aD)$, which is a map of operads with multiplication,
constructed as follows.

On arity $n$, the map takes the form
\begin{multline*}
\oEndb(\aC)(n)=\prod_{c_{0},\ldots,c_{n}}F(\aC(c_{1},c_{0})\sma \cdots \sma\aD(c_{n},c_{n-1}),\aD(c_{n},c_{0})) \\
\to\prod_{d_{0},\ldots,d_{n}}F(\aD(d_{1},d_{0})\sma \cdots \sma\aD(d_{n},d_{n-1}),\aD(d_{n},d_{0}))
=\oEndb(\aD)(n)
\end{multline*}
where on the $d_{0},\ldots,d_{n}$ factor of the target, we use
projection onto the $c_{i}=\phi(d_{i})$ factor of the source and
the isomorphism $\aC(\phi(d_{i}),\phi(d_{j}))\iso \aD(d_{i},d_{j})$
of the strictly fully faithful spectral functor $\phi$.

It is straightforward to verify that this map is compatible with the
operad structures and commutes with the inclusion of $\Ass$.

\begin{prop}\label{prop:limitednaturality}
The topological Hochschild-Mitchell construction $\CC$ extends to a
functor from the category of small spectral categories and
strictly fully faithful spectral functors to $\oD_{2}$-algebras
($E_{2}$ ring spectra).
\end{prop}

\section{Homotopy theory of objects and bimodules over small spectral categories}

Proposition~\ref{prop:limitednaturality} established a very limited
naturality for the functor $\CC$.  To extend this functoriality to the
more general derived statement of Theorems~\ref{main:dkfunct}
and~\ref{main:dcc}, we need to introduce in the next section some
conditions on bimodules that we
call centralizer conditions.  These are phrased in terms of the
derived functors of $\Hom^{\ell}_{\aC}$ and $\Hom^{r}_{\aD}$; the
purpose of this section is to set up the homotopical algebra and
review some conditions that ensure that the point-set functors
represent the derived functors.

Before beginning the discussion of model category structures, it is
convenient to introduce terminology for extending conditions and
properties on spectra and maps of spectra to small spectral categories, left
and right objects, and bimodules.

\begin{defnsch}
For any property or condition on spectra or maps of spectra, we say
the property holds \term{pointwise} on a bi-indexed spectrum, bi-indexed
spectrum with extra structure, or a map of bi-indexed spectra with
extra structure when it holds at every bi-index.  We say that such a
property holds \term{pointwise} for a small spectral category, bimodule,
strict morphism of small spectral categories, or map of bimodules
when it holds for the underlying bi-indexed
spectrum or map of bi-indexed spectra.
\end{defnsch}

For the model structures on the category of bi-indexed spectra and for
bi-indexed ring spectra $\aX$ and $\aY$, the categories of left
$\aX$-objects, right $\aY$-objects, and $(\aX,\aY)$-bimodules, we take
the weak equivalences to be the pointwise weak equivalences and the
fibrations to be the pointwise fibrations.  To describe the cofibrations,
let $I$ denote the standard set of generating cofibrations for the
model category of spectra.  Then given sets $A,B$, elements $a\in A$,
$b\in B$ and element $i\colon C\to D$ in $I$, let $C_{A,B;a,b;i}$ and
$D_{A,B;a,b;i}$ be the $(A,B)$-spectra that are $C$ and $D$
(respectively) on $(a,b)$ and $*$ elsewhere, and let
$f_{A,B;a,b;i}\colon C_{A,B;a,b;i}\to D_{A,B;a,b;i}$ be the map of
bi-indexed spectra that does $i\colon C\to D$ on $(a,b)$.  Although
the collection $BI$ of such maps does not form a small set, for any
given bi-indexed spectrum $\aZ$, the collection of maps from the
domains of the elements of $BI$ does form a small set (or is
isomorphic to one) since only those $C_{A,B;a,b;i}$ with $A=T(\aZ)$
and $B=S(\aZ)$ admit a map to $\aZ$.  This allows the small object
argument to be applied with the collection $BI$.  The cofibrations of
bi-indexed spectra are exactly the pointwise cofibrations (maps that
are cofibrations at each bi-index $(a,b)$) and these are exactly the
maps that are retracts of relative cell complexes built by attaching
cells from $BI$ (in the sense of \cite[5.4]{MMSS}).  The cofibrations
in the category of left $\aX$-objects, right $\aY$-objects, and
$(\aX,\aY)$ bimodules are the retracts of relative cell complexes (in
the sense of \cite[5.4]{MMSS}) built by attaching cells of the form
$\aX\otimes f_{O(\aX),B;a,b;i}$, $f_{A,O(\aY);a,b;i}\otimes \aY$, and
$\aX\otimes f_{O(\aX),O(\aY);a,b;i}$, respectively. The usual
arguments (e.g., \cite[\S5]{MMSS}) prove the following proposition.

\begin{prop}
The category of bi-indexed spectra
and for bi-indexed ring spectra $\aX$ and $\aY$, the categories of
left $\aX$-objects, right $\aY$-objects, and $(\aX,\aY)$-bimodules are
topologically enriched closed model categories with fibrations and weak
equivalences the pointwise fibrations and pointwise weak equivalences,
and cofibrations the retracts of relative cell complexes. The category
of $(\aX,\aY)$-bimodules is a spectrally enriched closed model category.
\end{prop}

\begin{prop}
For small spectral categories $\aC$ and $\aD$, the category of
$(\aC,\aD)$ bimodules is a spectrally enriched closed model category
with fibrations and weak equivalences the pointwise fibrations and
pointwise weak equivalences and cofibrations the retracts of relative
cell complexes.
\end{prop}

In the statement ``topologically enriched'' or ``spectrally enriched''
means that the categories satisfy the topological or spectral version
of Quillen's Axiom SM7, which is called the ``Enrichment Axiom'' in
\cite[\S3]{LewisMandell2}.  

For identifying cofibrant resolutions in
Section~\ref{sec:pfcatzigzag}, we need to know when cofibrant
bimodules are pointwise cofibrant; a sufficient condition is for the
small spectral categories in question to be pointwise ``semicofibrant'':
Recall from Lewis-Mandell~\cite[1.2,6.4]{LewisMandell2} that a
spectrum $X$ is \term{semicofibrant} when $X\sma(-)$ preserves
cofibrations and acyclic cofibrations, or equivalently, $F(X,-)$
preserves fibrations and acyclic fibrations.  Cofibrant spectra are in
particular semicofibrant; in the standard model structure on symmetric
spectra and orthogonal spectra, the sphere spectrum is cofibrant and
all semicofibrant objects are cofibrant.  In the positive stable model
category and in EKMM $S$-modules, the sphere spectrum is not cofibrant but
is only semicofibrant.  It follows formally that a weak equivalence of
semicofibrant spectra $X\to X'$ induces a weak equivalence $X\sma Y\to
X'\sma Y$ for any spectrum $Y$ (to see this, smash with a cofibrant
approximation of $\bS$) and a weak equivalence $F(X',Z)\to F(X,Z)$ for
any fibrant spectrum $Z$ \cite[6.2]{LewisMandell2}.  The explicit
description of cofibrations in the bimodule model structures implies
the following proposition.

\begin{prop}
A cofibrant bi-indexed spectrum is pointwise cofibrant.  If $\aX$ and
$\aY$ are pointwise semicofibrant bi-indexed ring spectra then the
cofibrant objects in left 
$\aX$-objects, right $\aY$-objects, and $(\aX,\aY)$-bimodules are
pointwise cofibrant.  In particular if $\aC$ and $\aD$ are pointwise
semicofibrant small spectral categories then cofibrant $(\aC,\aD)$-bimodules
are pointwise cofibrant.
\end{prop}

In later work, we use the following slightly stronger hypothesis on
the small spectral categories.

\begin{defn}\label{defn:relcof}
A small spectral category $\aC$ is pointwise relatively cofibrant when for
every object $c$ in $O(\aC)$, the unit map $\bS\to\aC(c,c)$ is a
cofibration of spectra, and for every pair of distinct objects $c,d$
in $O(\aC)$, the mapping spectrum $\aC(c,d)$ is cofibrant as a
spectrum.
\end{defn}

Pointwise relatively cofibrant small spectral categories are in particular
pointwise semicofibrant~\cite[1.3(c)]{LewisMandell2}.  
The following proposition produces sufficient examples of such
small spectral categories.  We made the following observation in~\cite[2.6--7]{BM-tc} based on
the earlier work of Schwede-Shipley~\cite[6.3]{SSMonoidalEq},
extending functoriality in strict morphisms to functoriality in
arbitrary spectral functors.  Although stated there in the context of 
symmetric spectra of simplicial sets, the same
arguments prove it for the other modern categories of spectra.

\begin{prop}\label{prop:replace}
Let $\aC$ be a small spectral category.  There are functorial spectral
categories $\aC^{\Cell}$ and $\aC^{\Cell,\Omega}$ and natural
DK-equivalences that are isomorphisms on object sets (or,
equivalently, strict morphisms that are pointwise weak equivalences)
\[
\aC\from \aC^{\Cell}\to \aC^{\Cell,\Omega}
\]
such that $\aC^{\Cell}$ is pointwise relatively cofibrant and
$\aC^{\Cell,\Omega}$ is pointwise relatively cofibrant and pointwise fibrant.
Moreover, if $\aC$ is pointwise fibrant, then so is $\aC^{\Cell}$.
\end{prop}

Next we move on to derived functors.  We concentrate on the case of
\[
\Hom^{\ell}_{\aX}(-,-)\colon (\LObj_{\aX})^{\op}\times \LObj_{\aX}\to \BiSp
\]
which takes a pair of left $\aX$-objects to a bi-indexed spectrum,
where $\aX$ is an arbitrary bi-indexed ring
spectrum.  The discussion for $\Hom^{r}$ has an exact parallel
for $\Hom^{\ell}$, switching left/right and source/target, with all
corresponding results holding.  In essence, Section~5
of~\cite{LewisMandell2} discusses this kind of derived functor,
although the story here is complicated because bi-indexed spectra and
categories of left 
objects are not enriched over spectra, but are only partially
enriched: once we fix a source set $A$, the full subcategory of left
$\aX$-objects with source set $A$ is isomorphic to the category of
$(\aX,\bS_{A})$-bimodules and then is enriched, while there are no
maps between left $\aX$-objects with different source sets.
Therefore, constructing ``partially enriched'' derived functors of two
variables for the entire
category of left $\aX$-objects is equivalent to constructing enriched
derived functors of two variables on each pair of these categories of
bimodules.  Applying \cite[5.8]{LewisMandell2}, we have the following result.

\begin{thm}\label{thm:derhom}
Let $\aX$ be a bi-indexed ring spectrum.  For the functor
\[
\Hom^{\ell}_{\aX}(-,-)\colon (\LObj_{\aX})^{\op}\times \LObj_{\aX}\to \BiSp
\]
the partially enriched right derived
functor $\bR\Hom^{\ell}_{\aX}(-,-)$ exists and is constructed by
cofibrant replacement of the first variable and fibrant replacement of
the second variable.
\end{thm}

\begin{proof}
As indicated above, we restrict the source set to $A$ for the first
variable (the contravariant variable) and the source set to $B$ for
the second variable (the covariant variable) to apply
\cite[5.8]{LewisMandell2} directly.  In the statement of
\cite[5.8]{LewisMandell2}, to reach this conclusion, we need to
observe that (1) $\Hom^{\ell}_{\aX}$ fits into an enriched parametrized
adjunction, (2) that each adjunction of one variable is a Quillen
adjunction when the parametrizing variable is cofibrant, and (3) that the
left adjoint preserves weak equivalences between cofibrant objects in
the parametrizing variable when the adjunction variable is cofibrant
(or, equivalently, the analogous condition for fibrant objects on the right adjoint). In this case, the 
enriched parametrized left adjoint is given by the
functor $\otimes$ that takes a left $\aX$-object with source set $A$
and an $(A,B)$-spectrum to a left $\aX$-object with source
set $B$. (Here the $(A,B)$-spectrum is the adjunction variable and the
left $\aX$-object with source set $A$ is the parametrizing variable.)
From the explicit description of cofibrations, it is clear
that $\otimes$ preserves cofibrations in each variable when the other
is cofibrant.  Since smash product of spectra preserves acyclic
cofibrations of spectra and the smash product with a cofibrant
spectrum preserves arbitrary weak equivalences, it is clear from the
formula for $\otimes$ that it preserves acyclic cofibrations in each
variable when the other variable is cofibrant.  This then verifies the
hypotheses of~\cite[5.8]{LewisMandell2}.
\end{proof}

When $\aX\to \aX'$ is a map of bi-indexed ring spectra, we obtain a
canonical forgetful or pullback functor from left $\aX'$-objects to
left $\aX$-objects, which induces a natural transformation
$\Hom^{\ell}_{\aX'}\to \Hom^{\ell}_{\aX}$ and a natural transformation
of derived functors $\bR\Hom^{\ell}_{\aX'}\to \bR\Hom^{\ell}_{\aX}$.
The argument for~\cite[8.3]{LewisMandell2} then implies the following
result.

\begin{prop}\label{prop:change}
If $\aX\to \aX'$ is a weak equivalence of bi-indexed ring spectra,
then the forget functor from left $\aX'$-objects to left $\aX$-objects
is the right adjoint of a Quillen equivalence and the
natural map $\bR\Hom^{\ell}_{\aX'}\to \bR\Hom^{\ell}_{\aX}$ is a
natural isomorphism in the homotopy category of bi-indexed spectra.
\end{prop}

We prefer to phrase the centralizer conditions in the next section in
terms of small spectral categories and bimodules over small spectral categories.
One technical wrinkle that arises (and indeed is the main issued
studied by \cite{LewisMandell2} as a whole) is that 
when we plug bimodules into $\Hom^{\ell}_{\aX}$ and consider functors
of the form
\begin{equation}\label{eq:bimodhom}
\Hom^{\ell}_{\aX}\colon (\Mod_{\aX,\aY})^{\op}\times \Mod_{\aX,\aZ},
\to \Mod_{\aY,\aZ}
\end{equation}
even when the enriched right derived functor exists, it may not agree
with the derived functor $\bR\Hom^{\ell}_{\aX}$ of
Theorem~\ref{thm:derhom} without hypotheses on $\aY$.  We return to
this question below.

The technical issue just mentioned causes some
awkwardness in trying to state a version of Theorem~\ref{thm:derhom}
for small spectral categories.  We dealt with this in the introduction by
phrasing the centralizer conditions in terms of homotopical bimodules,
which are defined as follows.  By neglect of structure, small spectral
categories become small categories enriched over the stable category;
in the definition below $\aD^{\op}\sma^{\bL}\aC$ denotes the small
category enriched over the stable category that is defined analogously
to $\aD^{\op}\sma \aC$ in the Definition~\ref{defn:sbimod}, but using the
smash product in the stable category.

\begin{defn}
A homotopical $(\aC,\aD)$-bimodule is an enriched
functor from $\aD^{\op}\sma^{\bL}\aC$ to the stable category.  More
generally, for a category $\aCat$ (partially) enriched over spectra or over the
stable category, homotopical left $\aC$-modules in $\aCat$, 
homotopical right $\aD$-modules in $\aCat$, and homotopical
$(\aC,\aD)$-bimodules in $\aCat$ are functors enriched over the stable
category from $\aC$, $\aD^{\op}$, and $\aD^{\op}\sma^{\bL}\aC$ into
$\aCat$, respectively.
\end{defn}

When $\aM$ is a $(\aC,\aD)$-bimodule, by neglect of structure it is a
homotopical right $\aD$-module in the homotopy category of left $\aC$-objects
and any cofibrant approximation in the category of left $\aC$-modules
inherits the canonical structure of a homotopical right $\aD$-module. Similar
observations apply to the fibrant approximation of a
$(\aC,\aE)$-bimodule, giving 
\begin{equation}\label{eq:derhom}
\bR\Hom^{\ell}_{\aC}(\aM,\aN)=\bR\Hom^{\ell}_{B_{\aC}}(B_{\aM},B_{\aN})
\end{equation}
the canonical structure of a homotopical $(\aD,\aE)$-bimodule, for
$\bR\Hom^{\ell}_{B_{\aC}}$ the right derived functor in Theorem~\ref{thm:derhom}.

We now return to the question of when the right derived functor
of~\eqref{eq:bimodhom} exists and is compatible with the right
derived functor in Theorem~\ref{thm:derhom}.  Although written in the
context of symmetric monoidal categories, Theorems~1.7(a) and~1.11(a)
of~\cite{LewisMandell2} show that both of these hold when the
underlying bi-indexed spectrum of $\aY$ is pointwise
semicofibrant.  In our context, the following gives the most convenient
statement; as always, the analogous result for $\Hom^{r}_{\aD}$ also holds. 

\begin{thm}\label{thm:LMmain}
Let $\aC$ and $\aD$ be small spectral categories and assume that $\aD$ is pointwise
semicofibrant. 
\begin{enumerate}
\item The forgetful functor from $(\aC,\aD)$-bimodules to left
$\aC$-objects preserves cofibrations (and all weak equivalences).
\item The enriched right derived functor of 
\[
\Hom^{\ell}_{\aC}\colon (\Mod_{\aC,\aD})^{\op}\times \Mod_{\aC,\aD}\to \Mod_{\aD,\aD}
\]
exists and is constructed by cofibrant replacement of the
contravariant variable and fibrant replacement of the covariant
variable.
\item Moreover, the underlying functor to homotopical 
$(\aD,\aD)$-bimodules of the right derived functor of~(ii) agrees with
the derived functor of~\eqref{eq:derhom}. 
\end{enumerate}
\end{thm}

\begin{proof}
The explicit description of the cofibrations shows that when
$\aD$ is pointwise semicofibrant, a cofibration of
$(\aC,\aD)$-bimodules forgets to a cofibration of left
$\aC$-objects. From here it is straightforward to check the conditions
of~\cite[5.4]{LewisMandell2} that ensure the existence of the enriched
right derived functor, and the comparison with the derived functor
of~\eqref{eq:derhom} is immediate.
\end{proof}

\section{Centralizer conditions, maps of $\CC$, and the proof of the
main theorem}\label{sec:pfstring}\label{sec:pfme}\label{sec:pfdcc}

In this section we begin the process of extending the functoriality of
$\CC$ by constructing zigzags associated to bimodules that satisfy
centralizer conditions that we review below.  We do enough work that
we can prove the main theorem of the introduction,
Theorem~\ref{main:string}, that gives an equivalence of $E_{2}$ ring
spectra for the two $\THC$ constructions commonly studied in string topology.
We also prove Theorems~\ref{main:me} and~\ref{main:dcc}.

We begin with the centralizer conditions.   

\begin{defn}\label{defn:dcc}
Let $\aC$ and $\aD$ be small spectral categories and let $\aF$ be a
$(\aC,\aD)$-bimodule.  The \term{centralizer map for $\aD$} is the map
in the category of homotopical $\aD$-modules
\[
\aD\to\bR\Hom^{\ell}_{\aC}(\aF,\aF)
\]
adjoint to the map $\aF\otimes\aD\to \aF$ where $\bR\Hom^{\ell}_{\aC}$
is as in~\eqref{eq:derhom} (that is, the right derived functor in
Theorem~\ref{thm:derhom}).  The \term{centralizer map for $\aC$} is
the analogous map
\[
\aC\to \bR\Hom^{r}_{\aD}(\aF,\aF).
\]
We say that:
\begin{enumerate}
\item $\aF$ satisfies the \term{double centralizer condition} when
both centralizer maps are weak equivalences.
\item $\aF$ satisfies the \term{single centralizer condition} for
$\aC$ or $\aD$ when the centralizer map for $\aC$ or $\aD$ (resp.) is
a weak equivalence.
\end{enumerate}
\end{defn}

We have the following motivating examples.

\begin{example}[DK-embeddings]\label{ex:dke}
If $\phi \colon \aD\to \aC$ is a spectral functor and $\aF$ is the
bimodule $\aF_{\phi}=\aC(\phi(-),-)$ then the enriched form of the Yoneda lemma
shows that $\Hom^{\ell}_{\aC}(\aF,\aF)$ is canonically isomorphic to
$\aC(\phi(-),\phi(-))$ and the centralizer map $\aD\to \Hom^{\ell}_{\aC}(\aF,\aF)$ is
the map $\phi\colon \aD(-,-)\to \aC(\phi(-),\phi(-))$; moreover,
$\aC(\phi(-),\phi(-))$ also represents 
the derived functor $\bR\Hom^{\ell}_{\aC}(\aF,\aF)$.  It follows that
$\aF$ satisfies the single 
centralizer condition for $\aD$ if and only if $\phi$ is a
DK-embedding.  Moreover, if $\phi$ is a DK-equivalence then the
enriched Yoneda lemma in the homotopy category shows that $\aC\to
\bR\Hom^{r}_{\aD}(\aF,\aF)$ is a weak equivalence and
$\aF$ satisfies the double centralizer condition.
\end{example}

\begin{example}[$\aD$ and $\perf(\aD)$]\label{ex:perf}
Let $\aD$ be a pointwise fibrant small spectral category, and $\aC$ be
a small full spectral subcategory of the category of right
$\aD$-modules consisting of only cofibrant-fibrant objects.  Assume
the Yoneda embedding factors $\phi
\colon \aD\to \aC$, and let $\aF=\aF_{\phi}$.  For example, $\aC=\perf(\aD)$ (for any large enough
cardinality) fits into this context.  Then the bimodule $\aF$
satisfies the double centralizer condition.  Since $\phi$ is a
DK-embedding, as per the previous example,  
$\aF$ satisfies the single centralizer condition for
$\aD$. To see that the centralizer map for $\aC$ is a weak
equivalence, we consider the map 
\[
\aC(x,y)\to \bR\Hom^{r}_{\aD}(\aF(-,x),\aF(-,y))
\]
for fixed $x,y$.  Recalling that $x$ and $y$ are $\aD$-modules, the
enriched Yoneda lemma gives isomorphisms
$x(d)\iso\aF(d,x)$ and $y(d)\iso\aF(d,y)$ for all $d$ in $\aD$, and
hence an isomorphism 
\[
\Hom^{r}_{\aD}(\aF(-,x),\aF(-,y))\iso \Hom^{r}_{\aD}(x(-),y(-))\iso\Mod_{\aD^{\op}}(x,y)=\aC(x,y)
\]
(see Proposition~\ref{prop:homover}).
Since we have assumed that $x$ and $y$ are cofibrant-fibrant right
$\aD$-modules, the point-set functor represents the right derived
functor, and we see that $\aF$ also satisfies the single centralizer
condition for $\aC$.
\end{example}

\begin{example}[Morita contexts]\label{ex:mc}
Let $\aM$ be a cofibrant $(\aC,\aD)$-bimodule (and we assume without
loss of generality that $\aC$ and $\aC$ are pointwise semicofibrant).
Then the left derived functor of $\aM\otimes_{\aD}(-)$ from the
derived category of $\aD$-modules
to the derived category of $\aC$-modules is an equivalence of homotopy
categories if and only
if $\aM\otimes_{\aD}(-)$ restricts to a DK-equivalence $\perf(\aD)\to
\perf(\aC)$ (for models of large enough cardinality).  When this
holds, the derived functor of the right adjoint
$\Hom^{\ell}_{\aC}(\aM,-)$ induces the inverse equivalence and
represents the right derived functor $\bR\Hom^{\ell}_{\aC}(\aM,-)$ in
Definition~\ref{defn:dcc}.  In particular, the unit of the
derived adjunction for $\aD$ is the centralizer map $\aD\to
\bR\Hom^{\ell}_{\aC}(\aM,\aM)$ and so $\aM$ satisfies the single
centralizer condition for $\aD$.  Although written in the context of
associative ring spectra, the proof of Theorem~4.1.2
of~\cite{SSStable} implies that there exists a cofibrant
$(\aD,\aC)$-bimodule $\aN$ such that $\aM\otimes_{\aD}\aN$ is weakly
equivalent to $\aC$ as a $(\aC,\aC)$-bimodule and
$\aN\otimes_{\aC}\aN$ is weakly equivalent to $\aD$ as a
$(\aD,\aD)$-bimodule.  It then follows that the left derived functor of
$(-)\otimes_{\aC}\aM$ from the derived category of right $\aC$-modules
to the derived category of right $\aD$-modules
is an equivalence of categories, which implies that $\aM$ satisfies
the single centralizer condition for $\aC$.  Thus, $\aM$ satisfies the
double centralizer condition.
\end{example}

\begin{example}[$DX$ and $\Omega X$]\label{ex:dxloop}
Let $X$ be a simply connected finite cell complex, or equivalently (up
to homotopy) the geometric realization of a reduced finite simplicial
set.  In~\cite[\S3]{BM-Koszul}, we consider the Kan loop group model
$GX$ for $\Omega X$ and describe an explicit $(\splus GX,DX)$-bimodule
$SP$ (whose underlying spectrum is equivalent to $\bS$) that we show
satisfies the double centralizer condition. 
(This example is originally due
Dwyer-Greenlees-Iyengar~\cite[4.22]{DwyerGreenleesIyengar-Duality}, at least
after extension of scalars to a field.)
\end{example}

To construct the zigzag, we use the following construction of
Keller~\cite[\S 4.5]{Keller-DIH}.

\begin{cons}
Let $\aC$ and $\aD$ be small spectral categories and $\aF$ a
$(\aC,\aD)$-bimodule.  Let $\Cat_{\aF}$ be the small spectral category with
objects $O(\aC)\amalg O(\aD)$, mapping spectra
\[
\Cat_{\aF}(a,b)=\begin{cases}
\aC(a,b)&a,b\in O(\aC)\\
*&a\in O(\aC), b\in O(\aD)\\
\aF(a,b)&a\in O(\aD), b\in O(\aC)\\
\aD(a,b)&a,b\in O(\aD)\\
\end{cases}
\]
with units coming from the units of $\aC$ and $\aD$, and composition
coming from the composition in $\aC$ and $\aD$ and the bimodule
structure of $\aF$.
\end{cons}

The construction comes with canonical strictly fully faithful spectral
functors $\aC\to \Cat_{\aF}$ and $\aD\to \Cat_{\aF}$, which by
Proposition~\ref{prop:limitednaturality} induce maps of
$\oD_{2}$-algebras
\[
\CC(\aD)\from \CC(\Cat_{\aF})\to \CC(\aC).
\]
The following theorem ties in the double centralizer condition.  We
prove it in Section~\ref{sec:pfcatzigzag}.

\begin{thm}\label{thm:catzigzag}
Assume $\aC$ and $\aD$ are pointwise relatively cofibrant and
pointwise fibrant small spectral categories and let $\aF$
be a pointwise semicofibrant-fibrant $(\aC,\aD)$-bimodule.
\begin{enumerate}
\item If $\aF$ satisfies the single centralizer condition for $\aD$, then
the map\break $\CC(\Cat_{\aF})\to \CC(\aC)$ is a weak equivalence. 
\item If $\aF$ satisfies the single
centralizer condition for $\aC$, then the map\break $\CC(\Cat_{\aF})\to
\CC(\aD)$ is a weak equivalence. 
\end{enumerate}
\end{thm}

If we take for granted that a functor $\THC$ exists as in
Theorem~\ref{main:wefunct}, then the previous theorem combined with
the examples above gives just what we need to prove
Theorems~\ref{main:string}, \ref{main:me}, and~\ref{main:dcc}.  

\begin{proof}[Proof of Theorem~\ref{main:string}] As per the statement
of Theorem~\ref{main:wefunct}, for any associative ring spectrum $A$,
$\THC(A)$ may be constructed as $\CC(A')$ for an associative ring
spectrum $A'$ whose underlying spectrum is fibrant and for which the
inclusion of the unit $\bS\to A$ is a cofibration of spectra (e.g.,
applying cofibrant and fibrant replacement functors in category of
associative ring spectra).  Indeed, in all previous literature
discussing $\THC(DX)$ and $\THC(\splus \Omega X)$, this was always
done tacitly.  Using such a model $DX'$ for $DX$ and $R$ for $\splus
\Omega X$ (or $GX$ as in Example~\ref{ex:dxloop}), we have a cofibrant
bimodule $SP$ satisfying the double centralizer condition, as in the
example.  The required chain of weak equivalences of $E_{2}$ ring
spectra is then given by the zigzag
\[
\CC(DX')\from \CC(\Cat_{SP})\to \CC(R)
\]
of weak equivalences of $\oD_{2}$-algebras.
\end{proof}

\begin{proof}[Proof of Theorems~\ref{main:me} and~\ref{main:dcc}]
By Examples~\ref{ex:dke} and ~\ref{ex:perf}, Theorem~\ref{main:me} is
a special case of
Theorem~\ref{main:dcc}.  The proof of Theorem~\ref{main:dcc} is
identical to the special case given by Theorem~\ref{main:string}:
Apply both parts of Theorem~\ref{thm:catzigzag} to appropriate
pointwise relatively cofibrant-fibrant
replacements as in Proposition~\ref{prop:replace}.
\end{proof}

\section{The construction of $\THC$ (Proof of
Theorems~\ref{main:wefunct} and~\ref{main:dkfunct})} 
\label{sec:pfwefunct}\label{sec:pfdkfunct}\label{sec:inftyfunct}

The purpose of this section is to construct topological Hochschild
cohomology as a homotopical functor.  We begin by constructing $\THC$
as a functor on the homotopy category level from a subcategory of the
homotopy category of small
spectral categories to the homotopy
category of $E_{2}$ ring spectra.  Using work of
Lindsey~\cite{Lindsey-XYZ}, we then show that essentially the same
argument actually constructs $\THC$ as a functor from a subcategory of
the $(\infty,1)$-category $\Cat^{\ex}$ of small stable idempotent-complete
$(\infty,1)$-categories to the $(\infty,1)$-category of $E_{2}$ ring spectra.
Throughout, we work with quasicategories as a model for
$(\infty,1)$-categories and rely on the foundational setup of Joyal
and Lurie~\cite{Lurie-HTT, Lurie-HA}.

\begin{defn}\label{defn:THCobj}
For a small spectral category $\aC$, let
$\THC(\aC)=\CC(\aC^{\Cell,\Omega})$ where $\aC^{\Cell,\Omega}$ is the
functorial pointwise relatively cofibrant-fibrant replacement of
Proposition~\ref{prop:replace}.
\end{defn}

Given a DK-embedding $\phi \colon \aD\to \aC$, by functoriality we get a
DK-embedding $\tilde\phi\colon \aD^{\Cell,\Omega}\to \aC^{\Cell,\Omega}$ and the
bimodule $\aF_{\tilde\phi}$ representing this functor
(see Definition~\ref{defn:bifunc}, Theorem~\ref{thm:functor}) satisfies the single centralizer
condition for $\aC^{\Cell,\Omega}$ (q.v.~Example~\ref{ex:dke}).
Writing $\Cat_{\tilde\phi}$ as an abbreviation for
$\Cat_{\aF_{\tilde\phi}}$, Theorem~\ref{thm:catzigzag} then gives us a
zigzag of maps of $\oD_{2}$-algebras 
\begin{equation}\label{eq:1simp}
\CC(\aC^{\Cell,\Omega})\overfrom{\simeq}\CC(\Cat_{\tilde\phi})\to \CC(\aD), 
\end{equation}
which we interpret as a map in the homotopy category of $E_{2}$ ring
spectra 
\[
\THC(\aC)\to \THC(\aD).
\]
This gives the next step in the construction of $\THC$ as a functor, the
definition on maps.

\begin{defn}\label{defn:THCmorph}
For a DK-embedding $\phi \colon \aD\to \aC$, define $\THC(\phi)$ to be
the map $\THC(\aC)\to \THC(\aD)$ in the homotopy category of
$E_{2}$-ring spectra arising from the zigzag of~\eqref{eq:1simp}.
\end{defn}

To check that this definition respects composition and unit maps, we
use the following construction.

\begin{defn}\label{defn:nsimp}
Let $\phi_{1}\colon \aC_{0}\to \aC_{1}$,\dots, $\phi_{n}\colon
\aC_{n-1}\to \aC_{n}$ be a composable sequence of spectral functors.
Define $\Cat_{\phi_{1},\ldots,\phi_{n}}$ to be the small spectral category
with objects the disjoint union of the objects of $\aC_{i}$ for all
$i$ and with mapping spectra 
\[
\Cat_{\phi_{1},\ldots,\phi_{n}}(a,b) =
\begin{cases}
\aC_{j}(\phi_{i,j}(a),b)&i\leq j\\
*&i>j
\end{cases}
\]
for $a\in O(\aC_{i})$ and $b\in O(\aC_{j})$, where $\phi_{i,j}=\id$ if
$i=j$ and $\phi_{i,j}=\phi_{j-1}\circ \cdots \circ \phi_{i}$ for
$i<j$.  Composition is induced by composition in $\aC_{0}$,\dots,
$\aC_{n}$ and the functors $\phi_{i}$, and units come from the units in $\aC_{0}$,\dots,
$\aC_{n}$.  
\end{defn}

We note that for a single morphism, $\Cat_{\phi}$ is
$\Cat_{\aF_{\phi}}$ for the bimodule $\aF_{\phi}$ associated to
$\phi$, which is consistent with the notation we used
in~\eqref{eq:1simp}.  We deduce from Theorem~\ref{thm:catzigzag} the
following corollary.

\begin{cor}\label{cor:nsimp}
With notation as in Definition~\ref{defn:nsimp}, assume each $\aC_{i}$
is pointwise relatively cofibrant-fibrant and that $\phi_{1}$ is
a DK-embedding.  Then the inclusion of $\aC_{0}$ in
$\Cat_{\phi_{1},\ldots,\phi_{n}}$ induces a weak equivalence
$\CC(\Cat_{\phi_{1},\ldots,\phi_{n}})\to \CC(\aC_{0})$.
\end{cor}

\begin{proof}
Let $\psi\colon \aC_{0}\to \Cat_{\phi_{2},\ldots,\phi_{n}}$ be the
composite of $\phi_{1}$ with the inclusion of $\aC_{1}$ in
$\Cat_{\phi_{2},\ldots,\phi_{n}}$.  We then have a canonical
isomorphism of small spectral categories from $\Cat_{\psi}$ to
$\Cat_{\phi_{1},\ldots,\phi_{n}}$.  Since $\phi_{1}$ is a
DK-embedding, so is $\psi$, and Theorem~\ref{thm:catzigzag} implies
that the induced map $\CC(\Cat_{\psi})\to \CC(\aC_{0})$ is a weak
equivalence.
\end{proof}

We can now prove Theorems~\ref{main:wefunct} and~\ref{main:dkfunct}.

\begin{proof}[Proof of Theorems~\ref{main:wefunct}
and~\ref{main:dkfunct}]
The proofs of the two theorems are essentially the same; for
the proof of Theorem~\ref{main:wefunct}, simply restrict to the
subcategory of small spectral categories consisting of the associative ring
spectra. (Note that even in the case of Theorem~\ref{main:wefunct}, the
argument still requires use of $\CC$ of small spectral categories, namely,
the small spectral categories $\Cat_{\aF_{\phi}}$.) 

We have defined $\THC$ on objects and morphisms in
Definitions~\ref{defn:THCobj} and~\ref{defn:THCmorph}; we need to show
that $\THC$ preserves composition and units.  Given $\phi_{1}\colon
\aC_{0}\to \aC_{1}$ and $\phi_{2}\colon \aC_{1}\to \aC_{2}$, let
$\tilde\phi_{1}$ and $\tilde\phi_{2}$ denote the induced functors on
$\aC_{i}^{\Cell,\Omega}$.  We then have the following strictly
commuting diagram of strictly fully faithful morphisms 
\[
\xymatrix@-1.25pc{%
&&\aC^{\Cell,\Omega}_{2}\ar[dl]\ar[dr]\\
&\Cat_{\tilde\phi_{2}\circ \tilde\phi_{1}}\ar[r]
&\Cat_{\tilde\phi_{1},\tilde\phi_{2}}
&\Cat_{\tilde\phi_{2}}\ar[l]\\
\aC^{\Cell,\Omega}_{0}\ar[rr]\ar[ur]&&\Cat_{\tilde\phi_{1}}\ar[u]&&\aC^{\Cell,\Omega}_{1}\ar[ll]\ar[ul]
}
\]
from which we get the following commutative diagram of
$\oD_{2}$-algebras.
\[
\xymatrix@-1.25pc{%
&&\CC(\aC^{\Cell,\Omega}_{2})\\
&\CC(\Cat_{\tilde\phi_{2}\circ \tilde\phi_{1}})\ar[ur]\ar[dl]_{\sim}
&\CC(\Cat_{\tilde\phi_{1},\tilde\phi_{2}})\ar[l]_{\sim}\ar[r]\ar[d]_{\sim}
&\CC(\Cat_{\tilde\phi_{2}})\ar[ul]\ar[dr]^{\sim}\\
\CC(\aC^{\Cell,\Omega}_{0})
&&\CC(\Cat_{\tilde\phi_{1}})\ar[ll]^{\sim}\ar[rr]
&&\CC(\aC^{\Cell,\Omega}_{1})
}
\]
The arrows marked with ``$\sim$'' are weak equivalences by
Theorem~\ref{thm:catzigzag}, Corollary~\ref{cor:nsimp}, and the
2-out-of-3 property.  Since $\THC(\phi_{1})$, $\THC(\phi_{2})$, and
$THC(\phi_{2}\circ \phi_{1})$ are defined by the outer zigzags in the
diagram above, we see that 
\[
\THC(\phi_{2}\circ \phi_{1})=\THC(\phi_{1})\circ \THC(\phi_{2}).
\]
Although $\THC(\id_{\aC})$ is not defined to be the identity map, part~(ii)
of Theorem~\ref{thm:catzigzag} shows that $\THC(\id_{\aC})$ is an
isomorphism (in the homotopy category), which together with the fact
just shown that $\THC(\id_{\aC})=\THC(\id_{\aC})\circ \THC(\id_{\aC})$, proves that
$\THC(\id_{\aC})$ is the identity map for any small spectral category $\aC$.
\end{proof}

Finally, we prove Theorem~\ref{main:inf} by explaining how to refine
$\THC$ into an functor of $\infty$-categories. 
For the source, for simplicity, we take the nerve of the category of
small spectral categories and DK-embeddings, $\aS\Cat^{DK}$; the
functor will take DK-equivalences (and indeed Morita
equivalences) to equivalences in the target, and so one can from there
factor through an $\infty$-categorical Bousfield localization.  For the target
category, we will use the homotopy coherent nerve of a
pointwise fibrant replacement of the Dwyer-Kan hammock localization of
the category of $\oD_{2}$-algebras, $N^{hc}L\aS[\oD_{2}]$.  We do not
get a point-set map of quasicategories, however, because although our
construction above takes morphisms of small spectral categories to zigzags
of $\oD_{2}$-algebras, which are honest morphisms in the hammock
localization, it does not preserve composition strictly.  If we think
in terms of zigzags in the original category of $\oD_{2}$-algebras,
the construction $\CC(\Cat_{\phi_{1},\ldots,\phi_{n}})$ gives
$n$-simplex zigzags associated to a sequence of composable
morphisms. Zachery Lindsey studied this kind of $\infty$-functoriality
in his 2018 Indiana University thesis~\cite{Lindsey-XYZ}; in the
notation there, we construct a map 
\[
N(\aS\Cat^{DK})\to \Zig(N^{hc}L\aS[\oD_{2}],N^{hc}L\aS[\oD_{2}]^{\simeq})
\]
as follows.
\begin{enumerate}
\setcounter{enumi}{-1}
\item A $0$-simplex of $N(\aS\Cat^{DK})$ is a small spectral category
$\aC$; it maps to $\CC(\aC^{\Cell,\Omega})$.
\item A $1$-simplex of $N(\aS\Cat^{DK})$ is a DK-embedding $\phi
\colon \aC_{0}\to \aC_{1}$; it maps to the zigzag~\eqref{eq:1simp}.
\item In general, an $n$-simplex consists of $n$-composable
DK-embeddings\break $\phi_{i}\colon \aC_{i-1}\to \aC_{i}$; it maps to
the $n$-simplex zigzag for $\Cat_{\phi_{1},\ldots,\phi_{n}}$
generalizing the $2$-simplex zigzag pictured in the proof of
Theorem~\ref{main:dkfunct}. 
\end{enumerate}

Although both the source and target simplicial sets are
quasicategories, this is not a map of quasicategories because it only
preserves face maps and not degeneracy maps.  The work of
Steimle~\cite{Steimle-Degeneracies} (see Theorems~1.2 and~1.4) allows
us to correct this to construct a functor from
$N(\aS\Cat^{DK})$ to $\Zig(N^{hc}L\aS[\oD_{2}],N^{hc}L\aS[\oD_{2}]^{\simeq})$.

Lindsey~\cite{Lindsey-XYZ} shows that the inclusion of a quasicategory $\aQ$ in the
quasicategory $\Zig(\aQ,\aQ^{\simeq})$ is a categorical equivalence.
This then proves the following theorem.

\begin{thm}
The preceding construction constructs a zigzag of maps of quasicategories
from $N\aS\Cat^{DK}$ to $N^{hc}L\aS[\oD_{2}]$,
providing a functor $\THC$ from the category of small
spectral categories and DK-embeddings to the category of $E_{2}$ ring
spectra that sends Morita equivalences to weak equivalences.
\end{thm}

Since $\THC$ sends Morita equivalences to weak equivalences, it
factors through the Bousfield localization of $N\aS\Cat^{DK}$ at the
Morita equivalences; using the equivalence of~\cite[4.23]{BGT} between
the localization of $N\aS\Cat$ at the Morita equivalences and
$\Cat^{\ex}$ then proves Theorem~\ref{main:inf}.

\section{Proof of Theorem~\ref{thm:catzigzag}}\label{sec:pfcatzigzag}

This section is devoted to the proof of Theorem~\ref{thm:catzigzag}.
The basic idea is to compare $\CC(\Cat_{\aF})$ to a construction of
the form $\Hom^{b}_{\aC,\aD}(\aR,\aF)$, where $\aR$ is a certain
simplicial object resolving $\aF$.  We start with the following
simplicial construction.

\begin{cons}
Let $\aC$ and $\aD$ be small spectral categories and 
let $\aG$ be a $(\aC,\aD)$-bimodule.  The simplicial
$(\aC,\aD)$-bimodule $\aR\subdot(\aC;\aG;\aD)$ is defined by
\[
\aR_{n}(\aC;\aG;\aD)=\bigvee_{j=0,\ldots,n+1}
\underbrace{\aC\otimes \cdots \otimes \aC}_{j\text{ factors}}
\otimes \aG\otimes 
\underbrace{\aD\otimes \cdots \otimes \aD}_{n+1-j\text{ factors}}
\]
(a total of $n+2$ summands each with $n+2$ factors) where the face map $d_{i}$ multiplies the
$i$th and $(i+1)$st factors using the multiplication of $\aC$ or $\aD$
or action on $\aG$ and the degeneracy map $s_{i}$ is induced by the map
$\bS_{O(\aC)}\to \aC$ in the $(i+1)$st factor on the $j$th summand for
$i<j$ and induced by the map $\bS_{O(\aD)}\to \aD$ in the $(i+1)$st factor on the $j$th summand for
$i\geq j$.  We write $\aR\subdot(\aG)$ when $\aC$ and $\aD$ are clear,
and we write $\aR(\aG)$ for the geometric realization.
\end{cons}

For example, $\aR_{0}(\aG)=\aG\otimes \aD\vee \aC\otimes \aG$ and the
degeneracy map $s_{0}$ is 
\[
\aG\otimes \aD\iso \aG\otimes \bS_{O(\aD)}\otimes \aD\to \aG\otimes \aD\otimes \aD
\]
on the $0$th summand and 
\[
\aC\otimes \aG\iso \aC\otimes \bS_{O(\aC)}\otimes \aG\to \aC\otimes \aC\otimes \aG
\]
on the $1$st summand.

We have an augmentation map of $(\aC,\aD)$-bimodules
$\epsilon \colon \aR\subdot(\aG)\to \aG$ induced by multiplying all the $\aC$ and
$\aD$ factors through.

\begin{prop}\label{prop:epsilon}
The augmentation $\epsilon \colon \aR\subdot(\aG)\to \aG$ is a
homotopy equivalence of simplicial bi-indexed spectra.
\end{prop}

\begin{proof}
In the category of bi-indexed spectra, the simplicial object
$\aR\subdot(\aG)$ has an ``extra degeneracy'' in the sense of
\cite[\S4.5]{Riehl-CategoricalBook}: Define $s_{-1}\colon
\aR_{n}(\aG)\to \aR_{n+1}\aG$ to be the map
\[
\aR_{n}(\aG)\iso \bS_{O(\aC)}\otimes \aR_{n}(\aG)\to \aC\otimes
\aR_{n}(\aG)\subset \aR_{n+1}(\aG).
\]
These maps satisfy
\begin{align*}
s_{-1}s_{i}&=s_{i+1}s_{-1}
&s_{-1}d_{i}&=d_{i+1}s_{-1}\\
s_{0}s_{-1}&=s_{-1}s_{-1}
&d_{0}s_{-1}&=\id.
\end{align*}
The map $s\colon \aG\to \aR_{0}(\aG)$ given by
\[
\aG\iso \bS_{O(\aC)}\otimes \aG\to \aC\otimes \aG\subset \aR_{0}(\aG)
\]
splits the map $\epsilon$ and (with $s_{-1}$) exhibits $\epsilon$ as
the split coequalizer of $d_{0},d_{1}\colon \aR_{1}(\aG)\to
\aR_{0}(\aG)$.  Meyer's theorem~\cite[4.5.1]{Riehl-CategoricalBook}
now gives the result.
\end{proof}

The Reedy model structures on simplicial and cosimplicial spectra are
convenient for identifying when maps of simplicial spectra realize to
cofibrations and when maps of cosimplicial spectra $\Tot$ to
fibrations.  The following proposition follows the usual outline of
similar results, which are proved from the pushout-product
property of the smash product of spectra and the construction of the
latching object of a simplicial spectrum as a sequence of pushouts.

\begin{prop}
If $\aC$ and $\aD$ are pointwise relatively cofibrant small spectral
categories and $\aG$ is a cofibrant $(\aC,\aD)$-bimodule then the
geometric realization of $\aR\subdot(\aG)$ is a cofibrant
$(\aC,\aD)$-bimodule and for every fibrant $(\aC,\aD)$-bimodule, the
cosimplicial spectrum $\Hom^{b}_{\aC,\aD}(\aR\subdot(\aG),\aF)$ is
Reedy fibrant. 
\end{prop}

The same kind of observation applied to the two-sided bar construction
$B(\aC;\aC;\aC)$ in the construction of $\CC$ proves
the following proposition.

\begin{prop}\label{prop:hypreedy}
If $\aC$ is a pointwise relatively cofibrant small spectral category
and $\aM$ is a fibrant $(\aC,\aC)$-bimodule, then the cosimplicial
spectrum $\CC\supdot(\aC,\aM)$ is Reedy fibrant.
\end{prop}

The previous proposition together with the formula for $\CC(\aC;\aM)$
in~\eqref{eq:HM} proves invariance under weak equivalences of fibrant
$\aM$ for $\aC$ satisfying the hypothesis.  Although this is all we
need for the proof of Theorem~\ref{thm:catzigzag} below, we state a
more general invariance theorem for convenience of future reference.

\begin{thm}\label{thm:inv}
Let $\aC$ be a pointwise relatively cofibrant small spectral
category and let $\aM$ be a a fibrant $(\aC,\aC)$-bimodule.
\begin{enumerate}
\item $\CC(\aC;\aM)$
represents the derived functor $\bR\Hom^{b}_{\aC,\aC}(\aC,\aM)$.  In
particular, $\CC(\aC;-)$ preserves weak equivalences between fibrant $(\aC,\aC)$-bimodules.
\item Assume $\aC'$ is a pointwise relatively cofibrant small spectral
category.  If $\phi \colon \aC'\to \aC$ be a DK-equivalence, then the
induced map $\CC(\aC;\aM)\to \CC(\aC';\phi^{*}M)$ is a weak equivalence.
\end{enumerate}
\end{thm}

\begin{proof}
The hypothesis on $\aC$ implies that the inclusion of
the degree zero part of bar construction 
\[
\aC\otimes \aC\to B\subdot(\aC;\aC;\aC)
\]
is a Reedy cofibration of $(\aC,\aC)$-bimodules, and it follows that
$B(\aC;\aC;\aC)$ is a semicofibrant $(\aC,\aC)$-bimodule.  Part~(i) is
then~\cite[6.3]{LewisMandell2}.  Part~(ii) follows immediately from
part~(i). 
\end{proof}

Proposition~\ref{prop:hypreedy} does not apply directly to
$\Cat_{\aF}$ under the hypotheses of Theorem~\ref{thm:catzigzag}
unless we further require $\aF$ to be pointwise cofibrant (which is
not the case in the main example of interest, $\aF=\aF_{\phi}$ for a
DK-embedding $\phi \colon \aD\to \aC$).  Nevertheless, the same
argument applies to prove the following proposition.

\begin{prop}
If $\aC$, $\aD$, and $\aF$ satisfy the hypotheses of
Theorem~\ref{thm:catzigzag}, then the cosimplicial spectrum
$\CC\supdot(\Cat_{\aF})$ is Reedy fibrant.
\end{prop}

\begin{cons}
Let $\aC$ and $\aD$ be small spectral categories and $\aF$ a
$(\aC,\aD)$-bimodule.  We construct a map of cosimplicial spectra 
\[
\gamma\supdot \colon \CC\supdot(\Cat_{\aF})\to \Hom^{b}_{\aC,\aD}(\aR\subdot(\aF),\aF)
\]
as follows.  In cosimplicial degree $0$,
we have
\[
\CC^{0}(\Cat_{\aF})\iso \prod_{d\in O(\aD)}\aD(d,d)\times \prod_{c\in O(\aC)}\aC(c,c)
\]
while 
\begin{align*}
\Hom^{b}_{\aC,\aD}&(\aR_{0}(\aF),\aF)=\Hom^{b}_{\aC,\aD}(\aF\otimes
\aD\vee \aC\otimes \aF,\aF)\\
&\iso \Hom^{b}_{\bS_{O(\aD)},\aD}(\aD,\Hom^{\ell}_{\aC}(\aF,\aF))
   \times \Hom^{b}_{\aC,\bS_{O(\aC)}}(\aC,\Hom^{r}_{\aD}(\aF,\aF))\\
&\iso \prod_{d\in O(\aD)}(\Hom^{\ell}_{\aC}(\aF,\aF))(d,d)\times 
\prod_{c\in O(\aC)}(\Hom^{r}_{\aD}(\aF,\aF))(c,c)
\end{align*}
and we define $\gamma^{0}$ to be the product of the centralizer maps.
For $n>0$, for any $j=1,\ldots,n$, given $c_{0},\ldots,c_{j-1}\in
O(\aC)$ and $d_{j},\ldots,d_{n}\in O(\aD)$ let 
\begin{multline*}
(\aC,\aF,\aD)_{n,j}(c_{0},\ldots,c_{j-1},d_{j},\ldots,d_{n})=\\
\qquad\quad  \aC(c_{1},c_{0})\sma \cdots \sma \aC(c_{j-1},c_{j-2})\sma
  \aF(d_{j},c_{j-1})\sma \aD(d_{j+1},d_{j})\sma\cdots\sma \aD(d_{n},d_{n-1}),
\end{multline*}
where we understand this formula as
\begin{align*}
&\aF(d_{1},c_{0})\sma \aD(d_{2},d_{1})\sma\cdots\sma
\aD(d_{n},d_{n-1}), \qquad \text{and}\\
&\aC(c_{1},c_{0})\sma \cdots \sma \aC(c_{n-1},c_{n-2})\sma
  \aF(d_{n},c_{n-1})
\end{align*}
when $j=1$ and $j=n$, respectively.   Then in this notation,
\begin{align*}
\CC^{n}&(\Cat_{\aF})\iso
\prod_{d_{0},\ldots,d_{n}\in O(\aD)}F(\aD(d_{1},d_{0})\sma\cdots\sma \aD(d_{n},d_{n-1}),\aD(d_{n},d_{0}))\\
&\times 
\prod_{j=1}^{n}\ \prod_{\putatop{c_{0},\ldots,c_{j-1}\in O(\aC)}{d_{j},\ldots,d_{n}\in O(\aD)}}
  F((\aC,\aF,\aD)_{n,j}(c_{0},\ldots,c_{j-1},d_{j},\ldots,d_{n}),\aF(d_{n},c_{0}))\\
&\times \prod_{c_{0},\ldots,c_{n}\in O(\aC)}F(\aC(c_{1},c_{0})\sma\cdots\sma \aC(c_{n},c_{n-1}),\aC(c_{n},c_{0}))
\end{align*}
while
\begin{align*}
\Hom^{b}_{\aC,\aD}&(\aR_{n}(\aF),\aF)\iso\\
&\prod_{d_{0},\ldots,d_{n}\in O(\aD)}F(\aD(d_{1},d_{0})\sma\cdots\sma \aD(d_{n},d_{n-1}),(\Hom^{\ell}_{\aC}(\aF,\aF))(d_{n},d_{0}))\\
&\times 
\prod_{j=1}^{n}\ \prod_{\putatop{c_{0},\ldots,c_{j-1}\in O(\aC)}{d_{j},\ldots,d_{n}\in O(\aD)}}
  F((\aC,\aF,\aD)_{n,j}(c_{0},\ldots,c_{j-1},d_{j},\ldots,d_{n}),\aF(d_{n},c_{0}))\\
&\times \prod_{c_{0},\ldots,c_{n}\in O(\aC)}F(\aC(c_{1},c_{0})\sma\cdots\sma \aC(c_{n},c_{n-1}),(\Hom^{r}_{\aD}(\aF,\aF))(c_{n},c_{0}))
\end{align*}
and we define $\gamma^{n}$ to be the map induced by the centralizer
maps $\aD\to \Hom^{\ell}_{\aC}(\aF,\aF)$ and $\aC\to
\Hom^{r}_{\aD}(\aF,\aF)$ on the outer factors and the identity on the
inner factors.  The maps $\gamma\supdot$ clearly commute with the
degeneracy maps and all but the zeroth and last face maps.  A tedious
but straightforward check of the definitions verifies that the
$\gamma\supdot$ also commutes with the zeroth and last face maps.
Let $\gamma$ denote the map on $\Tot$ induced by $\gamma\supdot$.
\end{cons}

\begin{proof}[Proof of Theorem~\ref{thm:catzigzag}]
Fix $\aC$, $\aD$, and $\aF$ as in the statement.
Let $\aF'\to \aF$ be a cofibrant replacement in the category of
$(\aC,\aD)$-bimodules, and consider the composite map
\[
\gamma'\colon \CC(\Cat_{\aF})\overto{\gamma}\Hom^{b}_{\aC,\aD}(\aR(\aF),\aF)\to 
\Hom^{b}(\aR(\aF'),\aF).
\]
We note that $\gamma'$ can be described in terms of the $\Tot$ of
a cosimplicial map with formula analogous to $\gamma$.
The inclusion of the summands 
\[
\aF'\otimes \aD \otimes \cdots \otimes \aD\qquad \text{and}\qquad 
\aC\otimes \cdots \otimes \aC\otimes \aF'
\]
in $\aR\subdot(\aF')$ assemble to a map of simplicial $(\aC,\aD)$-bimodules
\[
B\subdot(\aF';\aD;\aD)\vee
B\subdot(\aC;\aC;\aF)
\to \aR\subdot(\aC;\aF';\aD)
\]
where $B$ denotes the two-sided bar construction for $\otimes$.
The hypotheses on $\aC$, $\aD$, and $\aF'$ are sufficient for this map
to be a Reedy cofibration, and so it induces a fibration
\[
\Hom^{b}_{\aC,\aD}(\aR(\aF'),\aF)\to 
\Hom^{b}_{\aC,\aD}(B(\aF';\aD;\aD)\vee B(\aC;\aC;\aF'),\aF)
\]
since $\aF$ is fibrant.  We then have a canonical isomorphism 
\begin{align*}
&\Hom^{b}_{\aC,\aD}(B(\aF';\aD;\aD)\vee B(\aC;\aC;\aF'),\aF)\\
&\ \ \iso\Hom^{b}_{\aC,\aD}(B(\aF';\aD;\aD),\aF)\times \Hom^{b}_{\aC,\aD}(B(\aC;\aC;\aF'),\aF)\\
&\ \ \iso\Hom^{b}_{\aC,\aD}(\aF'\otimes_{\aD}B(\aD;\aD;\aD),\aF)\times \Hom^{b}_{\aC,\aD}(B(\aC;\aC;\aC)\otimes_{\aC}\aF',\aF)\\
&\ \ \iso\Hom^{b}_{\aD,\aD}(B(\aD;\aD;\aD),\Hom^{\ell}_{\aC}(\aF',\aF))\times \Hom^{b}_{\aC,\aC}(B(\aC;\aC;\aC),\Hom^{r}_{\aD}(\aF',\aF))\\
&\ \ \iso\CC(\aD;\Hom^{\ell}_{\aC}(\aF',\aF))\times \CC(\aC;\Hom^{r}_{\aD}(\aF',\aF))
\end{align*}
Opening up the construction of $\gamma$, we see that the following
diagram commutes
\[
\xymatrix{%
\CC(\Cat_{\aF})\ar[r]\ar[d]_{\gamma'}&\CC(\aD)\otimes \CC(\aC)\ar[d]\\
\Hom^{b}_{\aC,\aD}(\aR(\aF'),\aF)\ar[r]&\CC(\aD;\Hom^{\ell}_{\aC}(\aF',\aF))\times \CC(\aC;\Hom^{r}_{\aD}(\aF',\aF))
}
\]
where the right vertical map is induced by the double centralizer maps
on the bimodule variables of $\CC$.  We have observed that bottom
horizontal map is a fibration and in particular the $\Tot$ of a Reedy
fibration of cosimplicial spectra; the top horizontal map is also a
fibration and the $\Tot$ of a Reedy fibration of cosimplicial spectra.
The map on horizontal fibers is the $\Tot$ of the cosimplical map that
in each degree is the weak equivalence 
\begin{multline*}
\prod_{j=1}^{n}\ \prod_{\putatop{c_{0},\ldots,c_{j-1}\in O(\aC)}{d_{j},\ldots,d_{n}\in O(\aD)}}
  F((\aC,\aF,\aD)_{n,j}(c_{0},\ldots,c_{j-1},d_{j},\ldots,d_{n}),\aF(d_{n},c_{0}))\\
\to \prod_{j=1}^{n}\ \prod_{\putatop{c_{0},\ldots,c_{j-1}\in O(\aC)}{d_{j},\ldots,d_{n}\in O(\aD)}}
  F((\aC,\aF',\aD)_{n,j}(c_{0},\ldots,c_{j-1},d_{j},\ldots,d_{n}),\aF(d_{n},c_{0})).
\end{multline*}
$\Tot$ takes this degreewise weak equivalence of Reedy fibrant
objects to a weak equivalence of spectra.
It follows that the square above is homotopy cartesian. Both maps 
\begin{gather*}
\Hom^{b}(\aR(\aF'),\aF)\to 
\Hom^{b}_{\aC,\aD}(B(\aF';\aD;\aD),\aF)
\iso \CC(\aD;\Hom^{\ell}_{\aC}(\aF',\aF))\\
\Hom^{b}(\aR(\aF'),\aF)\to 
\Hom^{b}_{\aC,\aD}(B(\aC;\aC;\aF'),\aF)
\iso \CC(\aC;\Hom^{r}_{\aD}(\aF',\aF))
\end{gather*}
are weak equivalences since the maps $B(\aF';\aD;\aD)\to \aR(\aF)$ and
$B(\aC;\aC;\aF')\to \aR(\aF)$ are weak equivalences of cofibrant
$(\aC,\aD)$-bimodules.  Since $\aF'$ is cofibrant as both a left
$\aC$-object and right $\aD$-object (Theorem~\ref{thm:LMmain}.(i) and the
right object version), $\Hom^{\ell}_{\aC}(\aF',\aF)$ and
$\Hom^{r}_{\aD}(\aF',\aF)$ are pointwise fibrant.  It follows that when $\aD\to
\Hom^{\ell}_{\aC}(\aF',\aF)$ is a weak equivalence, so is the map
$\CC(\Cat_{\aF})\to \CC(\aC)$; likewise, 
when $\aC\to
\Hom^{r}_{\aD}(\aF',\aF)$ is a weak equivalence, so is the map
$\CC(\Cat_{\aF})\to \CC(\aD)$.  By Theorem~\ref{thm:LMmain}.(iii) (and its
analogue for $\Hom^{r}_{\aD}$), both 
$\Hom^{\ell}_{\aC}(\aF',\aF)$ and $\Hom^{r}_{\aD}(\aF',\aF)$ 
represent the derived functors in the statement of the centralizer
conditions.  The theorem now follows.
\end{proof}


\bibliographystyle{plain}
\bibliography{thhtc}

\end{document}